\documentclass{amsart}
\usepackage{graphicx, amsmath, amssymb, amsthm}
\usepackage{mathabx}
\usepackage{tikz-cd}
\usepackage{hyperref}
\usepackage{afterpage}
\usepackage{lscape}

\newcommand{\tpfstr}[2]{{\texorpdfstring{{#1}}{{#2}}}}

\usepackage[backend=biber, style=ieee-alphabetic, maxbibnames=99, doi=true,url=false,giveninits=true, hyperref]{biblatex}
\renewbibmacro{in:}{}

\newcommand{\NSFSupport}[1]{This material is based upon work supported by the National Science Foundation under Grant No. {#1}}

\newcommand{\MAHNSF}{DMS-1811189}

\newcommand{\MAHAddress}{University of California Los Angeles, Los Angeles, CA 90095}
\newcommand{\MAHemail}{\tt{mikehill@math.ucla.edu}}

\newcommand{\TLAddress}{University of Minnesota, Minneapolis, MN 55455}
\newcommand{\TLemail}{\tt{tlawson@umn.edu}}

\newcommand{\DSAddress}{University of Chicago, Chicago, IL 60637}
\newcommand{\DSemail}{\tt{dannyshixl@gmail.com}}

\newcommand{\MZAddress}{Utrecht University, Utrecht, 3584 CD, the Netherlands}
\newcommand{\MZemail}{\tt{mingcongzeng@gmail.com}}

\newcommand{\ABAddress}{University of Colorado Boulder, Boulder, CO 80309}
\newcommand{\ABemail}{\tt{agnes.beaudry@colorado.edu}}

\newcommand{\mycases}[1]{\left\{\begin{array}{ll}#1\end{array}\right.}
\newcommand{\N}{{\mathbb N}}
\newcommand{\F}{{\mathbb F}}

\newcommand{\co}{\colon\thinspace}
\newcommand{\RP}{{\mathbb{RP}}}
\newcommand{\T}{{\mathbb{T}}}
\renewcommand{\SS}{{\mathbb{S}}}

\newcommand{\HA}{A}

\newcommand{\aquot}{/\!/}

\newcommand{\HAn}[1][k]{\HA\langle #1\rangle}
\newcommand{\barHAn}[1][k]{\overline{\HA\langle#1\rangle}}
\newcommand{\tRn}[1][k]{{\tilde{\HA}\langle #1\rangle}}
\newcommand{\tRnast}[1][k]{{\tRn[#1]_{\ast}}}

\newcommand{\HAk}{{\HA\langle k\rangle}}
\newcommand{\barHAm}{{\overline{\HA\langle m\rangle}}}
\newcommand{\tRk}{{\tilde{\HA}\langle k\rangle}}
\newcommand{\tRkast}{{\tilde{\HA}\langle k\rangle_{\ast}}}

\newcommand{\Rkmn}[2]{\HAn[k]/(\zeta_{#1+1},\dots,\zeta_{#2})}

\newcommand{\RTwoFour}{{\tilde{\HA}\langle 2\rangle}}
\newcommand{\RThreeSix}{{\tilde{\HA}\langle 3\rangle}}

\newcommand{\Neg}{{\mhyphen}}

\newcommand{\Sksig}[1][k]{S(#1\sigma)}

\newcommand{\smashover}[1]{{\underset{#1}{\wedge}}}
\newcommand{\Boxover}[1]{\underset{#1}{\Box}}

\DeclareMathOperator{\Hom}{Hom}
\DeclareMathOperator{\Ext}{Ext}
\DeclareMathOperator{\Tor}{Tor}
\DeclareMathOperator{\Map}{Map}
\DeclareMathOperator{\End}{End}

\newcommand{\hocolim}{\mathrm{hocolim}}

\newcommand{\from}{\leftarrow}

\mathchardef\mhyphen=45

\numberwithin{equation}{section}

\newtheorem{theorem}{Theorem}[section]
\newtheorem{lemma}[theorem]{Lemma}
\newtheorem{corollary}[theorem]{Corollary}

\newtheorem{proposition}[theorem]{Proposition}

\newtheorem*{theorem*}{Theorem}
\newtheorem*{proposition*}{Proposition}

\theoremstyle{remark}
\newtheorem{remark}[theorem]{Remark}
\newtheorem{example}[theorem]{Example}
\newtheorem{notation}[theorem]{Notation}

\theoremstyle{definition}

\newtheorem{definition}[theorem]{Definition}

\title{Quotient rings of \(H\F_2 \wedge H\F_2\)}

\author[AB]{Agn\`es Beaudry}
\address{\ABAddress}
\email{\ABemail}

\author[MAH]{Michael A.~Hill}
\address{\MAHAddress}
\email{\MAHemail}

\author[TL]{Tyler Lawson}
\address{\TLAddress}
\email{\TLemail}

\author[XDS]{XiaoLin Danny Shi}
\address{\DSAddress}
\email{\DSemail}

\author[MZ]{Mingcong Zeng}
\address{\MZAddress}
\email{\MZemail}

\thanks{\NSFSupport{DMS-1906227 and \MAHNSF}}

\begin{document}

\maketitle
\begin{abstract}
We study modules over the commutative ring spectrum \(H\mathbb F_2\wedge H\mathbb F_2\), whose coefficient groups are quotients of the dual Steenrod algebra by collections of the Milnor generators. We show that very few of these quotients admit algebra structures, but those that do can be constructed simply: killing a generator \(\xi_k\) in the category of associative algebras freely kills the higher generators \(\xi_{k+n}\). Using new information about the conjugation operation in the dual Steenrod algebra, we also consider quotients by families of Milnor generators and their conjugates. This allows us to produce a family of associative \(H\mathbb F_2\wedge H\mathbb F_2\)-algebras whose coefficient rings are finite-dimensional and exhibit unexpected duality features. We then use these algebras to give detailed computations of the homotopy groups of several modules over this ring spectrum.
\end{abstract}

\section{Introduction}

In stable homotopy theory, the mod-2 Moore spectrum \(M\) is a ``quotient'' of the sphere that does not admit a unital multiplication. This has several proofs: one uses mod-2 cohomology, while another observes that \(\pi_* M\) does not admit the structure of a ring. As a result, this particular example is often employed to illustrate a barrier between stable homotopy theory and more classical algebra. It is tempting to blame these problems on the non-algebraic nature of the stable homotopy category itself, the difficult nature of the coefficient ring \(\pi_* \SS\), general unpleasant behavior of the prime \(2\), or similar factors. Examples of similar phenomena have been found in stable module categories by Langer \cite{LangerStableModule} and in exotic triangulated categories by Muro--Schwede--Strickland \cite{MuroSchwedeStricklandTriangulated}.

If \(H\F_2\) is the mod-2 Eilenberg--Mac Lane spectrum, the smash product
\[
\HA = H\F_2 \wedge H\F_2
\]
is a commutative ring spectrum. The coefficient ring is Milnor's dual Steenrod algebra
\[
{\HA_{\ast}} \cong \F_2[\xi_1, \xi_2, \dots],
\]
a polynomial algebra on generators \(\xi_i\) in degree \(2^i-1\) \cite{milnor}. We denote by \(\zeta_i\) the conjugate of \(\xi_i\).

The spectrum \(\HA\) is an algebra over \(H\F_2\). By work of Shipley \cite{ShipleyAlgebras}, this implies that it is equivalent to a differential graded algebra over \(\F_2\). Even further, it can be shown that it is formal: this differential graded \(\F_2\)-algebra is equivalent to the graded polynomial algebra \({\HA_{\ast}}\) with zero differential, and as an associative \(H\F_2\)-algebra \(\HA\) is equivalent to the generalized Eilenberg--Mac Lane spectrum \(H(\HA_\ast)\). This gives the category of left \(\HA\)-modules a concrete algebraic model: the category of differential graded \({\HA_{\ast}}\)-modules.

However, these equivalences ignore commutativity in the multiplicative structure: \(\HA\) is not equivalent to \(H(\HA_\ast)\) as a commutative algebra, and these two rings give incompatible monoidal structures on the category of left \(\HA\)-modules. Our goal in this paper is to study the gap between the theory of algebras and modules over \(\HA\) and \({\HA_{\ast}}\). We will find that examples behaving like the mod-2 Moore spectrum are abundant, despite left \(\HA\)-modules being very algebraic and despite the coefficient ring \(\HA_\ast\) being very tame.

\subsection*{Overview and results}
 One of our chief approaches in this paper is to use the relative homology 
 \[
 H_*^\HA(M) = \pi_*(H\F_2 \smashover{\HA} M)
 \]
 and the corresponding relative Adams spectral sequence. These relative groups give us an indispensable tool for the study of \(\HA\)-modules and \(\HA\)-algebras.
 
 Section~\ref{sec:mappingcones} is dedicated to these topics. In \S\ref{sec:relASS}, we discuss this relative Adams spectral sequence: Baker--Lazarev's \(H \F_2\)-based Adams spectral sequence in the category of \(\HA\)-modules \cite{BakerLazarev}. Computing with it requires an understanding of relative homology groups as comodules over the Hopf algebra 
 \[
 \Gamma=\pi_*(H \F_2 \smashover{\HA} H\F_2),
 \]
 which we will identify with B\"okstedt's topological Hochschild homology.

In \S\ref{sec:homofmappcones}, we compute the relative homology of quotients of \(\HA\) by classes in its homotopy, and describe their comodule structure over \(\Gamma\). We find that, homotopically, coning off a class does not kill it. This brings us to our first main result, which we prove in \S\ref{sec:double}.

\begin{theorem}\label{thm:introcone}
  Given any element \(\alpha \in {\HA_{\ast}}\) such that \(\alpha \equiv \xi_i\) mod decomposable elements, with mapping cone \(C(\alpha)\), the multiplication-by-\(\alpha\) map induces the zero map on \(\pi_* C(\alpha)\) but not on \(\pi_*(C(\alpha) \smashover{\HA} C(\alpha))\). In particular, \(\alpha\) is not the trivial self-map of \(C(\alpha)\) as an \(\HA\)-module.
\end{theorem}

As in the case of the mod-\(2\) Moore spectrum, a consequence of this calculation is that these cones \(C(\alpha)\) have no \(\HA\)-linear multiplication. It may, at first, seem that Theorem~\ref{thm:introcone} is at odds with our algebraic description of the category of left modules. However, the multiplication-by-\(\alpha\) map \(M \to M\) does not naturally have the structure of a map of left modules without assuming some commutativity from \(\HA\). 

Iterated cones are related to endomorphism algebras in \S\ref{sec:endo}, and we compute the relative Adams \(E_2\)-page for these and other related endomorphism objects.

While Section~\ref{sec:mappingcones} is mostly focused on modules, in Section~\ref{sec:killing} we consider the effect of coning off classes in the category of associative algebras. We begin with general definitions and results about associative algebra quotients in \S\ref{sec:killinggens}: given an element \(\alpha \in \pi_* \HA\), there is a universal associative \(\HA\)-algebra \(\HA\aquot\alpha\) with a chosen nullhomotopy of \(\alpha\). In \S\ref{sec:apptoAmod}, we apply the relative Adams spectral sequence to calculate with these associative algebra quotients of \(\HA\). By computing the relative homology and Adams \(E_2\)-term, we obtain a surprising result that identifies the associative algebra quotient \(\HA\aquot \xi_{k+1} \) with one of the iterated cones \(\HA/(\xi_{k+1}, \xi_{k+2}, \ldots) \) from Section~\ref{sec:mappingcones}. 

\begin{theorem}
The universal associative \(\HA\)-algebra \(\HAn := \HA\aquot\xi_{k+1}\) with a chosen nullhomotopy of \(\xi_{k+1}\) has, as coefficient ring, the following \(\HA_\ast\)-algebra:
\[
  \HAn_\ast = {\HA_{\ast}}/(\xi_{k+1},\xi_{k+2},\dots).
\]
Similarly, \(\barHAn := \HA\aquot\zeta_{k+1}\) has the following coefficient ring:
\[
  \barHAn_\ast = {\HA_{\ast}}/(\zeta_{k+1},\zeta_{k+2},\dots).
\]
\end{theorem}
This follows from a general result that we prove only later in Appendix~\ref{sec:CupkAlgs}.

\begin{theorem}\label{thm:kernelq1}
Suppose that \(B\) has a multiplication in \(\HA\)-modules with unit \(\eta\co \HA \to B\). Then the ideal
  \[
    \ker(\eta) \subset {\HA_{\ast}}
  \]
must be closed under the Dyer--Lashof operation \(Q_1\).
\end{theorem}
This is perhaps unexpected: the operation \(Q_1\) usually does not have to be preserved by \(\eta\) because it is not even defined on \(B\). This result very strongly constrains the ideals that can appear as kernels of quotients because the operation \(Q_1\) on \({\HA_{\ast}}\) is highly nontrivial by work of Steinberger \cite[\S III.2]{Hinfinity}. 

The \(C_2\)-action on \(\HA\) makes the quotient algebras \({\HAn}\) and \(\barHAn\) conjugate to each other, and this lets us extend \(\HA\) to a \(C_2\)-equivariant algebra \({\HAn} \smashover{\HA} \barHAn\) that we would like to understand for future applications to equivariant homotopy theory. As a prerequisite to understanding further quotient rings of \(\HA\) like these, Section~\ref{sec:regsteenrodquotients} (which is largely independent) addresses purely algebraic questions about the dual Steenrod algebra \(\HA_*\) and its quotients. In it, we prove the following nontrivial fact about the conjugation operation.
\begin{theorem}\label{thm:introregu}
  In the quotient \({\HAn_{\ast}}\) of the dual Steenrod algebra, the sequence of elements 
  \[\zeta_{m+1}, \zeta_{m+2}, \dots, \zeta_{m+k}\]
  is regular for any \(m \geq 0\), and all higher \(\zeta_i\) are in the ideal generated by them.
  
  Consequently, \(\HAn_\ast / (\zeta_{m+1},\dots,\zeta_{m+k})\) is a commutative, graded Frobenius algebra over \(\F_2\) with top class in degree \(2^{m+k+1} - 2^{m+1} - 2^{k+1} + 2.\)
\end{theorem}
Our proof of the regularity of these sequences will require some commutative algebra; our efforts to give this fact a simple proof using Milnor's formula for conjugation have not been successful.

In Section~\ref{sec:splitting}, we return to homotopy theory and the study of associative algebras under \(\HAk\smashover{\HA}\barHAm\). The surprise continues: many quotients by a finite sequence of classes \(\zeta_i\) beyond the regularity range still admit ring structures. 

\begin{theorem}
For any natural numbers \(k, m, n\geq 0\), there is an associative algebra structure on the iterated mapping cone
  \[
    \Rkmn{m}{m+k+n},
  \]
  and these all receive associative \(\HA\)-algebra maps from
  \[
    \HAn\smashover{\HA}\barHAm.
  \]
    When \(n=0\), the coefficient ring of \(   \Rkmn{m}{m+k}\) is the quotient
  \[
    \F_2[\xi_1,\dots,\xi_k]/(\zeta_{m+1},\dots,\zeta_{m+k}).
  \]
  When \(n > 0\), there is a splitting into indecomposable \(\HA\)-modules:
  \[
    \Rkmn{m}{m+k+n} \simeq \Rkmn{m}{m+k} \wedge \left(\bigvee_{j=0}^{2^n-1} S^{j(2^{m+k+1})}\right)
  \]
\end{theorem}

In fact, these quotients \( \Rkmn{m}{m+k+n}\) arise as Postnikov stages of \(\HAn \smashover{\HA} \barHAm\). The splitting into suspensions of \(\Rkmn{m}{m+k}\) allows us to conclude that their relative homology and relative Adams spectral sequences are completely determined by that of of \(\Rkmn{m}{m+k}\).

In Section~\ref{sec:duality} we discuss duality. We first isolate a few facts about connective associative ring spectra whose coefficient ring is a Frobenius algebra, which turn out to include the quotients \( \Rkmn{m}{m+k+n}\). The duality on their coefficient rings has consequences, including spectral Gorenstein duality and duality for the coefficient groups of modules.

\begin{theorem}
    Fix any natural numbers \(k,m,n\geq 0\), and let  \[
    \delta = \delta(k,m,n)= 2^{m+k+n+1} - 2^{m+1} - 2^{k+1} + 2.
    \]
    For any dualizable \(\HA\)-module \(M\), the \(\HA\)-module Spanier--Whitehead dual \(D_\HA M\) has an isomorphism of \(\HAn_\ast/(\zeta_{m+1},\dots,\zeta_{m+k})\)-modules:
    \[
    \pi_\ast\!\big(\Rkmn{m}{m+k+n} \smashover{\HA} D_A M\big) \cong\Big( \pi_{\delta-\ast}\!\big(\Rkmn{m}{m+k+n} \smashover{\HA} M\big)\Big)^\vee
    \]
    Here the right-hand side is the \(\F_2\)-linear dual.
    In particular,
    \[
    \pi_d\big(\Rkmn{m}{m+k+n}\big) \cong\Big( \pi_{\delta-d}\big(\Rkmn{m}{m+k+n}\big)\Big)^\vee.
    \]
\end{theorem}

This fact implies that any iterated quotient of \(\Rkmn{m}{m+k}\) by classes in the homotopy of \(\HA\) is self-dual up to a shift. 

In Section~\ref{sec:Examples}, we illustrate the effectiveness of the theory by computing the relative Adams spectral sequence for a series of examples. Our first result in \S\ref{sec:ExampleRkn} is to prove a family of differentials in the Adams spectral sequence of \(\HAk/(\zeta_{m+1},\dots,\zeta_{m+k})\).  In \S\ref{sec:Atilde2} we completely compute the relative Adams spectral sequence of \(\HAn[2]/(\zeta_3,\zeta_4)\), in \S\ref{sec:AllButXiTwo} that of \(\HAn[2]/(\xi_1, \zeta_1, \zeta_3,\zeta_4)\), and in \S\ref{sec:AllButXiThree} that of \(\HAn[3]/(\xi_1, \zeta_1, \xi_2,\zeta_2, \zeta_4, \zeta_5, \zeta_6)\). In these examples, duality proves to be extremely useful for solving extension problems.

Since it is technical and thematically significantly different from the rest of the paper, we have left the proof of Theorem~\ref{thm:kernelq1}, about the Dyer--Lashof operation \(Q_1\), to Appendix~\ref{sec:CupkAlgs}. We start in \S\ref{sec:cupkfirst} by defining cup-\(k\) algebras and discussing an inductive approach to understanding them. In \S\ref{sec:cupkops} we review the associated power operations. In \S\ref{sec:cupkcentrality} we define a notion of centrality for a cup-\((k-1)\) algebra over a cup-\(k\) algebra and explore a few consequences of this definition. Finally, in \S\ref{sec:cupktransgression}, we prove that for a central for a cup-\((k-1)\) algebra  \(B\) over a a cup-\(k\) algebra \(A\), the kernel of \(\pi_*A \to \pi_*B\) is closed under power operations.

\subsection{Acknowledgements}

The authors would like to thank
Ian Coley,
Drew Heard,
Mike Hopkins,
Lennart Meier,
Haynes Miller,
and
Doug Ravenel
for discussions related to this work.

\section{The relative Adams spectral sequence and mapping cones}\label{sec:mappingcones}

\subsection{The relative Adams spectral sequence}\label{sec:relASS}

The key computational tool we use is Baker--Lazarev's relative Adams spectral sequence based on \(H\F_2\)-homology in the category of \(\HA\)-modules \cite{BakerLazarev}. For this, we need to compute the \(\HA\)-module version of the dual Steenrod algebra. In our case, however, this is due to B\"okstedt \cite{BokstedtTHH}; see also Franjou--Lannes--Schwartz for a published reference \cite{FranjouLannesSchwartz}. 

\begin{theorem}\label{thmBokstedt}
  The commutative ring spectrum \(H\F_2\smashover{\HA} H\F_2\) is the topological Hochschild homology of \(H\F_2\). The coefficient ring is
  \[
  \pi_\ast \big(H\F_2\smashover{\HA} H\F_2\big)\cong
  \F_2[{u}],
  \]
  where \(|{u}|=2\).
\end{theorem}

Since \(\pi_\ast\big(H\F_2\smashover{\HA} H\F_2\big)\) is flat over \(\pi_\ast H\F_2\), the pair 
\[
(A,\Gamma)=\big(\F_2,\F_2[{u}]\big)
\]
forms a Hopf algebra, and \(u\) is primitive for degree reasons. By Baker--Lazarev's extension of Adams' original argument, the relative homology
\[
H_\ast^\HA(M) = \pi_\ast(H\F_2 \smashover{\HA} M)
\]
of any \(\HA\)-module is a comodule over this Hopf algebra, and we have a relative Adams spectral sequence. This particular relative Adams spectral sequence has been employed in equivariant homotopy theory by Hahn--Wilson \cite{HahnWilson}.

\begin{corollary}\label{cor:AdamsSS}
  For any \(\HA\)-modules \(M\) and \(N\), there is a relative Adams spectral sequence with 
  \[
  E_{2}^{s,t}=\Ext_{\F_2[{u}]}^{s,t}\!\big(H_\ast^\HA (M),H_\ast^\HA (N)\big),
  \]
  converging (conditionally) to the homotopy classes of \(\HA\)-module maps from \(M\) to \(N\).
\end{corollary}

When we take \(M=N=\HA\), giving the relative Adams spectral sequence for the homotopy of \(\HA\), this takes the form
\[
\Ext_{\F_2[{u}]}^{s,t}(\F_2, \F_2) \Rightarrow \F_2[\xi_1, \xi_2, \dots],
\]
The dual to the coalgebra \(\F_2[u]\) is an exterior algebra on duals to the primitives \(u^{2^i}\). The \(E_2\)-term is therefore polynomial on classes 
\[
[{u}^{2^{k-1}}]\in\Ext^{1,2^k}_{(\F_2,\F_2[{u}])}(\F_2,\F_2),
\]
and comparing with the abutment, we see that the spectral sequence collapses at \(E_2\). The class \([{u}^{2^{k-1}}]\) detects the homotopy element \(\xi_{k}\). This gives us an algebraic interpretation of the relative Adams filtration: the relative Adams filtration of \(\HA_\ast\) is the filtration by monomial degree in the ungraded sense, while the topological degree records the fact that \(\HA_\ast\) is a graded ring.

More generally, for any \(\HA\)-module \(M\) the relative Adams spectral sequence
\[
\Ext_{\F_2[u]}^{s,t}\!\big(\F_2, H_\ast^\HA(M)\big) \Rightarrow \pi_\ast M
\]
is always a module over the relative Adams spectral sequence for \(\HA\) itself, and this module structure recovers the associated graded of the action of \({\HA_{\ast}}\) on \(\pi_* M\).

\subsection{Homology of Mapping Cones}\label{sec:homofmappcones}
We can now apply the relative Adams spectral sequence machinery of \S\ref{sec:relASS} to study the effect of killing various classes in the homotopy of \(\HA\).

\begin{definition}
  For any \(\alpha \in {\HA_{\ast}} = \pi_* \HA\), let \(C(\alpha) = \HA/\alpha\) denote the mapping cone of the self-map \(\Sigma^{|\alpha|} \HA \to \HA\).
  
  Let \(M_k\) be the \(\HA\)-module \(C(\xi_k) = \HA/\xi_k\). More generally, for any subset \(I\) of the positive natural numbers, let
  \[
  M_I:=\bigwedge_{i\in I}^{\HA}M_i.
  \]
\end{definition}
There is a natural map \(\HA \to M_I\), and as an \(\HA_\ast\)-module, we have
\[
\pi_\ast(M_I) \cong \HA_\ast / (\xi_i \mid i \in I).
\]
This implies that
\(
M_{\mathbb N_{>0}}\cong H\F_2,
\)
and, more generally, we can view any of the modules \(M_{I}\) as augmented to \(H\F_2\).

We can compute the relative homology of any of these cones fairly directly. Recall a key observation about the relative Adams filtration: if \(f\colon N\to M\) is a map of relative Adams filtration at least \(1\), then the relative homology of the cone on \(f\), \(C(f)\), sits in an exact sequence
\[
0\to H_\ast^{\HA}(M)\to H_\ast^{\HA}\big(C(f)\big)\to \Sigma H_\ast^{\HA}(N) \to 0.
\]
Moreover, this extension of comodules is the extension corresponding to the class in \(\Ext^1\) detecting \(f\). If \(f\) has relative Adams filtration at least \(2\), then the extension is trivial. 
The classes \({u}^{2^{k-1}}\) are the primitives in the Hopf algebra \(\Gamma\), which generically makes them appear in \(\Ext^1\). The following proposition expands upon this when \(N\) is a shift of \(\HA\).

\begin{proposition}\label{prop:muzeroNull}
Let \(M\) be an \(\HA\)-module and \(y_n\in \pi_nM\) with relative Hurewicz image \(\bar{y}_n \in H_n^\HA(M)\). Let \(\alpha \in \pi_{k-1} \HA\) be any element that maps to zero in \(\F_2\), and consider the cofiber sequence
\[
\Sigma^{k+n-1} \HA \xrightarrow{\alpha y_n} M \to C(\alpha y_n)
\]
so that \(C(\alpha y_n)\) is the mapping cone on \(\alpha y_n\) in \(\HA\)-modules. Then
\begin{enumerate}
    \item the relative homology of \(C(\alpha y_n)\) is given by
    \[
    H_\ast^{\HA}\big(C(\alpha y_n)\big)\cong H_\ast^{\HA}(M)\oplus\F_2\cdot\bar{x}_{k+n},
    \]
    where \(|\bar{x}_{k+n}|=k+n\), and 
    \item the coaction is determined by 
    \[
    \psi(\bar{x}_{k+n})=\begin{cases}
        1\otimes\bar{x}_{k+n}+{u}^{2^{n-1}}\otimes \bar{y}_n & \text{if }\alpha\equiv \xi_n\mod\text{ decomposables,} \\
        1\otimes \bar{x}_{k+n}& \text{otherwise.}
        \end{cases}
    \]
\end{enumerate}
\end{proposition}

\begin{corollary}\label{cor:MkRelHomology}
  The relative homology of \(M_k\), as a comodule, is 
  \[    
  H_\ast^\HA(M_k)\cong \F_2\{b_0,b_{2^k}\},
  \]
  with \(|b_{i}|=i\), where the coaction is given by
  \[
  \psi(b_{2^k})=1\otimes b_{2^k}+{u}^{2^{k-1}}\otimes b_0
  \]
  and \(b_0\) is primitive.
\end{corollary}

For any of our modules \(M_k\), the natural maps 
\[
\HA\to M_k \to H\F_2
\]
are isomorphisms on \(\pi_0\). On relative homology, these maps become
\[
\F_2 \to \F_2\{b_0, b_{2^k}\} \to \F_2[u].
\]
The class \(b_0\) is the image of \(1\) under the first map, while the second map sends \(b_0\) to \(1\) and \(b_{2^k}\) to \({u}^{2^{k-1}}\).

The general K\"unneth spectral sequence for modules over a ring spectrum, defined in Elmendorf--Kriz--Mandell--May \cite[Theorem IV.4.1]{EKMM}, then allows us to deduce the more general case, since everything is flat over \(\F_2\).

\begin{corollary}
  If \(I\) is a subset of the positive natural numbers, then the relative homology of \(M_I\), as a comodule, is
  \[
  H_\ast^\HA(M_I)\cong\bigotimes_{i\in I} \F_2\{b_0,b_{2^i}\},
  \]
  with \(|b_j|=j\) and where the coaction is the one on the  tensor product induced by the Hopf algebra structure.
\end{corollary}

Note that in this tensor product, each factor has a primitive element \(b_0\) in degree zero and a unique non-primitive element \(b_{2^i}\). Each of these non-primitive elements is in a distinct degree, so we can rewrite this using dyadic expansions.

\begin{notation}\label{not:Dyadic}
If \(I\) is a subset of the natural numbers, then let 
\[
D_{I}=\left\{\sum_{i \in S} 2^i \mid S \subset I\text{ finite}\right\}
\]
be the set of natural numbers whose dyadic expansions only involve terms of the form \(2^i\) with \(i\in I\).
\end{notation}

\begin{corollary}\label{cor:HomologyOfMI}
  If \(I\) is a subset of the positive natural numbers, then the relative homology of \(M_I\) is
  \[
  H_\ast^\HA(M_I)\cong
  \F_2\{b_j\mid j\in D_I\},
  \]
  with \(|b_j|=j\). The coproduct is given by
  \[
  \psi(b_j)=\sum_{i\in D_{I}} \binom{j}{i} {u}^{\big(\frac{j-i}{2}\big)}\otimes b_{i}.
  \]
  The augmentation
  \[
  H_\ast^\HA(M_I) \to H_\ast^\HA(H\F_2) = \F_2[u]
  \]
  is injective, and sends \(b_j\) to \(u^{j/2}\).
\end{corollary}

This gives an algebraic ``no-go theorem'' for ring structures. Although the homotopy groups  \(\pi_* M_I\) form a quotient ring of \(\HA_\ast\), these spectra typically cannot have \(\HA\)-algebra structures.

\begin{theorem}
  Let \(I\) be a subset of the positive natural numbers. If \(M_I\) is a unital ring object in the homotopy category of \(\HA\)-modules, then whenever \(i\in I\) we must have \((i+1)\in I\). In particular, if \(I\) is nonempty and finite there can be no such ring structure.
\end{theorem}
\begin{proof}
  The class \(b_{2^i} \in H_{2^i}^\HA(M_I)\) has coproduct
  \[
  \psi(b_{2^i})=1\otimes b_{2^i}+{u}^{2^{i-1}}\otimes b_0,
  \]
  and hence if there be a ring structure, we must have 
  \[
  \psi(b_{2^i}^2)=1\otimes b_{2^i}^2+{u}^{2^{i}}\otimes b_0^2.
  \]
  The element \(b_0\) is the only possible image of the unit, and hence the square must be again \(b_0\), making this coproduct non-zero. This forces \(b_{2^i}^2 \neq 0\). However, the only possible non-zero class in degree \(2^{i+1}\) would be \(b_{2^{i+1}}\), which exists if and only if \((i+1) \in I\).
\end{proof}

\begin{remark}
    Our identification of the comodule structure in Proposition~\ref{prop:muzeroNull} shows that the comodule structure depends only on the image of a class modulo decomposables. The same ``no-go theorem'' holds for any sequence of classes 
    \[
    \{\chi_k\mid k\in I, \chi_k\equiv \xi_k\mod\text{ decomposables}\}.
    \]
\end{remark}
This remark applies to the conjugate classes \(\zeta_k\), but there is also a general procedure for studying conjugates.

\begin{definition}
  If \(M\) is an \(\HA\)-module, then let \(\overline{M}\) denote the pullback of \(M\) along the automorphism \(\HA\to\HA\) induced by the swap on \(H\F_2 \wedge H\F_2\).
\end{definition}

\begin{example}
  For any subset \(I\) of the positive natural numbers, \(\overline{M}_{I}\) is the quotient of \(\HA\) by the classes \(\zeta_i\) with \(i\in I\).
\end{example}

We pause to connect the coaction on a module \(M\) to that on \(\overline{M}\). 

\begin{proposition}\label{prop:MvsBarM}
  For any \(\HA\)-module \(M\), we have a natural isomorphism of \(\F_2[{u}]\)-comodules
  \[
  H_\ast^\HA (M)\cong
  H_\ast^\HA (\overline{M}).
  \]
\end{proposition}
\begin{proof}
  In this instance, it can be clarifying to discuss the general case first. For a ring homomorphism \(f\co R \to S\), an \(S\)-algebra \(E\), and an \(S\)-module \(M\), the map \(f\) induces a commutative diagram
  \[
  \begin{tikzcd}
    E \smashover{R} f^*(M) \ar[r] \ar[d] &
    (E \smashover{R} E) \smashover{E} (E \smashover{R} f^*(M)) \ar[d] \\
    E \smashover{S} M \ar[r] &
    (E \smashover{S} E) \smashover{E} (E \smashover{S} M).
  \end{tikzcd}
  \]
  We will assume \(\Gamma_R = \pi_*(E \smashover{R} E)\) and \(\Gamma_S = \pi_*(E \smashover{S} E)\) are flat over \(\pi_* E\). On homotopy groups, this becomes a commutative diagram
  \[
  \begin{tikzcd}
    E_\ast^R(f^*(M)) \ar[r, "\psi"] \ar[d] &
    \Gamma_R \otimes_{\pi_* E} E_\ast^R(f^*(M)) \ar[d] \\
    E_\ast^S(M) \ar[r,"\psi"] &
    \Gamma_S \otimes_{\pi_* E} E_\ast^S(M).
  \end{tikzcd}
  \]
  The map \(\bar f\co (\pi_* E, \Gamma_R) \to (\pi_* E, \Gamma_S)\) induced by \(f\) is a map of Hopf algebroids, and so this diagram expresses the existence of a natural map of comodules
  \[
  (\bar f)_* \left(E_\ast^R( f^* M)\right) \to E_\ast^S (M)
  \]
  over \((\pi_* E, \Gamma_S)\). If \(f\) is an equivalence, then this natural map of comodules is an isomorphism.
  
  Now take \(f\) to be the swap automorphism of \(\HA\), so that \(f^* M = \overline{M}\) by definition. We get an isomorphism of \(\F_2[{u}]\)-comodules
  \[
  (\bar f)_* H_\ast^\HA\big(\overline M\big) \to H_\ast^\HA(M).
  \]
  In this case, however, the conjugation automorphism \(f\) of \(\HA\) induces the trivial automorphism of \(\F_2[{u}]\): the conjugation in \({\HA_{\ast}}\) is the identity mod decomposables. Therefore, \((\bar f)_*\) is the identity.
\end{proof}

\begin{remark}
    We see the algebraic isomorphism of Proposition~\ref{prop:MvsBarM} reflected in an identification between the corresponding relative Adams spectral sequences. This is simply recording that modulo elements of higher monomial degree, \(\xi_i\equiv \zeta_i\).
\end{remark}

\subsection{Double killing classes resurrects them}\label{sec:double}
One of the most curious features of this framework is that iterated quotients by classes often see multiplication by those classes reappear. This was an early observation that led us to this general framework, starting with the guess that in \(M_1\smashover{\HA}\overline{M}_1\), multiplication by \(\xi_1\) is non-trivial. 

The following results on comodule structure follow from the calculation of the relative homology of \(M_k\) by the K\"unneth theorem.
\begin{proposition}\label{prop:DoubleKilling}
We have isomorphisms
\[
H_\ast^\HA\big(M_k \smashover{\HA} M_k\big) \cong H_{\ast}^\HA\big(M_k\smashover{\HA}\overline{M}_k\big)\cong \F_2\{b_0\otimes b_0, b_{2^k}\otimes b_0, b_0\otimes b_{2^k}, b_{2^k}\otimes b_{2^k}\},
\]
and the coproduct on \(b_{2^k}\otimes b_{2^k}\) is
\[
\psi(b_{2^k}\otimes b_{2^k})=1\otimes b_{2^k}\otimes b_{2^k}+{u}^{2^{k-1}}\otimes \big(b_{2^k}\otimes b_0+b_0\otimes b_{2^k}\big)+{u}^{2^k}\otimes b_0\otimes b_0.
\]
\end{proposition}
\begin{proof}
Proposition~\ref{prop:MvsBarM} shows that the relative homologies of \(M_k\) and \(\overline{M}_k\) agree as comodules. The result follows from the K\"unneth theorem.
\end{proof}

Note that since both the left and right copies of \(b_{2^k}\) in Proposition~\ref{prop:DoubleKilling} support the same coproduct, the sum
\[
e_{2^k}:=b_{2^k}\otimes b_0+b_0\otimes b_{2^k}
\]
is actually a primitive element in the comodule. Of course, the element \(e_0:=b_0\otimes b_0\) is also primitive.

\begin{remark}The calculation of the comodule structure here is essentially the same as the calculation that the smash square of the mod \(2\) Moore spectrum has a non-trivial \(Sq^2\) connecting the bottom and top cells. In the category of \(\HA\)-modules, this is part of a family.
\end{remark}

An immediate consequence of Proposition~\ref{prop:DoubleKilling} is that we have a non-trivial multiplication by \(\xi_k\) in the relative Adams spectral sequence for \(M_k\smashover{\HA}\overline{M}_k\).

\begin{corollary}\label{cor:AdamsMkBarMk}
The relative Adams \(E_2\)-term for \(M_k \smashover{\HA} M_k\) or \(M_k\smashover{\HA}\overline{M}_k\) is
\[
\F_2[\xi_1,\dots]\{e_0, e_{2^k}\}/(\xi_k\cdot e_0, \xi_k\cdot e_{2^k}+\xi_{k+1}\cdot e_{0}),
\]
where the bidegree \((t-s,s)\) of \(e_{j}\) is \((j,0)\) and that of \(\xi_i\) is \((2^i-1,1)\).
\end{corollary}

\begin{remark}\label{rem:CofiberForm}
    Either module \(M=M_k \smashover{\HA} M_k\) or \(M=M_k\smashover{\HA}\overline{M}_k\) sits in a cofiber sequence
    \[
    M_k\to M \to \Sigma^{2^k} M_k.
    \]
    The relative Adams spectral sequence for \(M_k\) collapses at \(E_2\) to
    \[
    \F_2[\xi_1,\dots]\{e_0\}/\xi_k\cdot e_0,
    \]
    and similarly for the \(2^k\)-fold suspension. Corollary~\ref{cor:AdamsMkBarMk} says that this is a non-trivial extension on relative Adams \(E_2\)-terms.
\end{remark}

Using Remark~\ref{rem:CofiberForm}, we can also determine the homotopy. The connecting map in the long exact sequence on homotopy induced by this cofiber sequence is exactly multiplication by \(\xi_k\) or \(\zeta_k\) in the module
\[
\F_2[\xi_1,\dots] \{e_0\}/\xi_k\cdot e_0.
\]
In the former case, the map is identically zero and the spectral sequence collapses at \(E_2\). In the latter, Milnor's formula tells us
\[
\sum_{i=0}^k \xi_i^{2^{k-i}}\zeta_{k-i}=0,
\]
so modulo \(\xi_k\), we find two cases:
\begin{enumerate}
    \item if \(k=1\), then \(\zeta_k=\xi_k\equiv 0\), while
    \item if \(k>1\), then \(\zeta_k\equiv \xi_1\xi_{k-1}^2+\dots\), where the elided terms are elements of monomial degree at least \(5\).
\end{enumerate}

\begin{corollary}\label{cor:AdamsdThree}
    In the relative Adams spectral sequence for \(M_k\smashover{\HA}\overline{M}_k\), when \(k > 1\) we have a \(d_3\)-differential:
    \[
    d_3(e_{2^k})=\xi_1\xi_{k-1}^2\cdot e_0
    \]
    The spectral sequence collapses at \(E_4\).
\end{corollary}
\begin{proof}
    The image \(\xi_1 \xi_k^2 \cdot e_0\) of \(\zeta_k \cdot e_0\) must be zero on the \(E_\infty\) page, and the only class available to support this differential is \(e_{2^k}\). It remains only to show that the spectral sequence collapses at \(E_4\). Multiplication by \(\xi_1\xi_{k-1}^2\) is injective in the  module
    \[
    \F_2[\xi_1,\dots]\{e_0\}/(\xi_k\cdot e_0),
    \]
    so the \(E_4\)-page is
    \[
    \F_2[\xi_1,\dots]\{e_0\}/(\xi_k, \xi_1\xi_{k-1}^2)\cdot e_0.
    \]
    For degree reasons, \(e_0\) is a permanent cycle, and therefore all multiples of \(e_0\) are as well.
\end{proof}

The elided terms in the formula for \(\zeta_k\) will reappear as multiplicative extensions: the class \(\xi_1\xi_{k-1}^2\) is only the leading term of \(\zeta_k\) modulo \(\xi_k\). We will expand on this in \S\ref{sec:ExampleRkn}.

\subsection{Endomorphism algebras of cones}\label{sec:endo}

The endomorphism ring of a module is a natural example of an associative algebra. When we consider finite cell \(\HA\)-modules, the underlying homotopy type can be easily determined, such as in the case of the algebras
\[
\End_{\HA}(M_I)
\]
when \(I\) is a finite set. 

Since the cones \(M_k\) are self dual in the sense that the \(\HA\)-module Spanier--Whitehead dual \(D_\HA M_k\) has an equivalence
\[
D_{\HA} M_k\simeq \Sigma^{-2^k} M_k,
\]
we have an equivalence of \(\HA\)-module spectra:
\[
\End_\HA(M_k) \simeq \Sigma^{-2^k} M_k \smashover{\HA} M_k
\]

\begin{notation}
  We will write \(E(S)\) for the exterior algebra on a graded set \(S\).
\end{notation}

\begin{proposition}
    The homotopy of \(\End_{\HA}(M_k)\), as an \(\HA_\ast\)-algebra, is
    \[
    \HA_\ast\otimes E(\beta_{\Neg2^k})/(\xi_k\beta_{\Neg2^k}, r_k),
    \]
    where the relation \(r_k\) satisfies
    \[
    r_k\equiv\xi_k+\xi_{k+1}\beta_{\Neg2^k}\mod\text{ higher order monomials}.
    \]
\end{proposition}
\begin{proof}
Corollary~\ref{cor:AdamsMkBarMk} says that the \(E_2\)-term of the relative Adams spectral sequence is the following:
\[
\HA_\ast \otimes E(\beta_{\Neg2^k}) / (\xi_k\beta_{\Neg2^k}, \xi_{k+1} \beta_{\Neg2^k} + \xi_k),
\]
where \(\beta_{\Neg2^k}\) is in degree \((-2^k)\). For degree reasons, \(\beta_{\Neg2^k}\) is necessarily exterior and a permanent cycle.

This spectral sequence degenerates at \(E_2\) because all generators of this algebra are permanent cycles. In \(\pi_\ast \End_\HA(M_k)\), therefore, \(\xi_k\) is divisible by \(\xi_{k+1}\) mod terms that are quadratic and higher in \(\HA_\ast\).
\end{proof}

\begin{remark}
One consequence of this theorem is that the ``generating hypothesis'', as in Lockridge \cite{LockridgeGeneratingHyp}, fails in the category of \(\HA\)-modules: there are nontrivial self-maps of the finite complex \(M_k\) that induce the zero map on homotopy groups. 
\end{remark}

When \(k=1\), there is no indeterminacy in the relation \(r_1\): the relation \(\xi_1\beta_{\Neg2}=0\) implies that in degree \(1\), the only possible relation \(r_1\) must be that \(\xi_1=\xi_2\beta_{\Neg2}\). Working inductively, we get a family of algebras where we have good control: the algebras \(\End_\HA(M_{\leq k})\).

\begin{proposition}
The relative Adams spectral sequence for \(\End_\HA(M_{\leq k})\) collapses at \(E_2\), and there is an isomorphism of algebras
\[
\pi_\ast \End_\HA(M_{\leq k}) \cong 
\F_2[\xi_{k+1}, \xi_{k+2}, \dots] \otimes E(\beta_{\Neg2},\dots,\beta_{\Neg2^k}).
\]
For \(j \leq k\), the image of \(\xi_j\) in this algebra is \(\xi_{j+1} \beta_{\Neg2^j}\).
\end{proposition}

\begin{proof}
We use the equivalence \(M_{\leq k}=M_1\smashover{\HA}  \ldots \smashover{\HA} M_k \) to deduce that
\[
\End_\HA(M_{\leq k}) \simeq \End_\HA(M_1) \smashover{\HA} \dots \smashover{\HA} \End_\HA(M_k).
\]
The homology
\[
H_\ast^\HA(\End_\HA(M_k)) \cong \Sigma^{2-2^{k+1}} H_\ast^\HA(M_{\leq k}) \otimes H_\ast^\HA(M_{\leq k})
\]
is cofree over the quotient coalgebra \(\F_2[u]/(u^{2^k}) \cong H_\ast^\HA(M_{\leq k})\). The Cartan-Eilenberg spectral sequence then degenerates, and so the \(E_2\)-term of the relative Adams spectral sequence is
\[
\Ext_{\F_2[u^{2^k}]} (\F_2, \Sigma^{2-2^{k+1}} H_\ast^\HA(M_{\leq k})) \cong \F_2[\xi_{k+1}, \xi_{k+2}, \dots] \otimes E(\beta_{\Neg2},\dots,\beta_{\Neg2^k}).
\]
This again collapses at \(E_2\): the maps \(\End_\HA(M_j) \to \End_\HA(M_{\leq k})\) show, by naturality of the relative Adams spectral sequence, that the elements \(\beta_{\Neg2^j}\) are permanent cycles that square to zero and that the extensions \(\xi_j = \xi_{j+1} \beta_{\Neg2^j}\) persist to this \(E_2\)-term.

The first nonzero element higher than the \(1\)-line is \(\xi_{k+1}^2 \beta_{\Neg2} \beta_{\Neg4} \dots \beta_{\Neg2^k}\) in total degree \(2^{k+1}\). This means that there are no classes in the same total degree as \(\xi_j + \xi_{j+1} \beta_{\Neg 2^j}\) of higher relative Adams filtration, so the relations on \(E_2\) are in fact exact.
\end{proof}

\begin{proposition}
For all \(k\), there are equivalences
\[
\End_\HA(M_{\leq k}) \smashover{\HA} M_{k+1} \simeq \End_\HA(M_{\leq k}) \smashover{\HA} \overline M_{k+1}.
\]
\end{proposition}

\begin{proof}
The previous proposition shows that, for \(j \leq k\), the image of the element \(\xi_j\) in \(\pi_\ast \End_\HA(M_{\leq k})\) is \(\xi_{j+1} \beta_{\Neg 2^j}\), and this squares to zero. The image of Milnor's formula for conjugation then shows that \(\xi_{k+1}\) and \(\zeta_{k+1}\) become equal in \(\pi_\ast \End_\HA(M_{\leq k})\), and so their mapping cones are equivalent as left modules over \(\End_\HA(M_{\leq k})\).
\end{proof}

\begin{corollary}
There is an equivalence of \(\HA\)-modules
\[
M_{\leq k} \smashover{\HA} \overline M_{\leq k} \simeq \Sigma^{2^{k+1}-2} \End_\HA(M_{\leq k}).
\]
\end{corollary}

\begin{proof}
 This follows by induction on \(k\).
\end{proof}

\section{Killing \tpfstr{\(\xi_i\)}{Xi-i}s in associative algebras}\label{sec:killing}
The ring \(\HA = H\F_2 \wedge H\F_2\) is an associative ring spectrum, so we can universally kill off elements in homotopy by attaching \(A_\infty\)-cells. Doing so {\emph{in the category of \(\HA\)-algebras}} has an added wrinkle that arises from ensuring that the new null-homotopy commutes with \(\HA\). Using results from Appendix~\ref{sec:CupkAlgs}, we will see that killing an element in homotopy in associative algebras has the effect of killing a swath of Dyer--Lashof operations on that element.

\subsection{Killing generators}\label{sec:killinggens}

The method of killing generators here is closely related to work of Baker--Richter and Szymik \cite{BakerRichterQuasisymmetric, SzymikCharacteristic}.

\begin{definition}
  For a commutative ring spectrum \(R\) and an \(R\)-module \(M\), we define \(\T_R(M)\) to be the free associative \(R\)-algebra on \(M\):
  \[
  \T_R(M) = \bigvee_{n \geq 0} M^{\smashover{R} n}.
  \]
\end{definition}

\begin{definition}
  Suppose that \(R\) is a commutative ring spectrum and \(\alpha \in \pi_n(R)\), represented by a map \(\Sigma^n R \to R\). The \emph{associative quotient} \(R\aquot\alpha\) is the (homotopy) pushout
  \[
    \begin{tikzcd}
      \T_R(\Sigma^n R) \ar[r] \ar[d, swap, "\alpha"] &
      \T_R(C \Sigma^n R) \ar[d] \\
      R \ar[r] & R\aquot\alpha
    \end{tikzcd}
  \]
  in the category of associative \(R\)-algebras. Here \(C \Sigma^n R\) is the cone \(\Sigma^n R \wedge I \simeq \ast \).
\end{definition}

\begin{proposition}
  The algebra \(R\aquot\alpha\) is the free associative \(R\)-algebra on a nullhomotopy of \(\alpha\) in the following sense. For an associative \(R\)-algebra \(B\), a map \(R\aquot\alpha \to B\) of associative \(R\)-algebras is equivalent to a choice of nullhomotopy of the composite map \(\Sigma^n R \stackrel{\alpha}{\to} R \to B\) of \(R\)-modules.
\end{proposition}

\begin{proposition}\label{prop:aquotbasechange}
  For a map \(f\co R \to S\) of commutative ring spectra and \(\alpha \in \pi_* R\), there is an equivalence
  \[
  S \smashover{R} (R\aquot\alpha) \simeq S\aquot f(\alpha)
  \]
  of associative algebras over \(S\).
\end{proposition}

\begin{example}\label{ex:killingzero}
  A choice of nullhomotopy of the zero map \(S^n \to B\) is equivalent to a map \(S^{n+1} \to B\). As a result, if \(\alpha = 0\) in \(\pi_* R\), then the universal property shows that \(R \aquot \alpha\) is a free algebra:
  \[
  R \aquot \alpha \cong \T_R(\Sigma^{n+1} R) \simeq \bigvee_{k \geq 0} \Sigma^{k(n+1)}R.
  \]
  More generally, if \(f\co R \to S\) is a map of commutative ring spectra such that \(f(\alpha) = 0\), then this shows that there is an equivalence
  \[
  S \smashover{R} (R \aquot \alpha) \simeq \T_S(\Sigma^{n+1} S)
  \]
  of \(S\)-algebras.
\end{example}

\subsection{Applications to \tpfstr{\(\HA\)}{A}-modules}\label{sec:apptoAmod}

\begin{proposition}\label{prop:HomologyAmodmodalpha}
  Suppose that \(\alpha \in {\HA_{\ast}}\) is such that \(\alpha \equiv \xi_{k+1}\) mod decomposables. Then there is an isomorphism 
  \[
  \pi_*(H\F_2 \smashover{\HA} A\aquot \alpha) \cong \F_2[{u}^{2^k}]
  \]
  of comodules over \(\F_2[{u}]\).
\end{proposition}

\begin{proof}
  The element \(\alpha\) maps to zero in \(H\F_2\). By Example~\ref{ex:killingzero}, there is an equivalence of \(H\F_2 \smashover{\HA} (\HA\aquot \alpha)\) with the free algebra \(\T_{H\F_2}(S^{2^{k+1}})\) . In particular,
  \[
  H_\ast^\HA(\HA \aquot \alpha) = \pi_*(H\F_2 \smashover{\HA} \HA\aquot \alpha) \cong \F_2[b_{2^{k+1}}].
  \]
  
  The universal property of \(\HA\aquot \alpha\) gives it a map \(\HA\aquot \alpha \to H\F_2\) of \(\HA\)-algebras, factoring the map \(C(\alpha) \to H\F_2\). Applying \(H_\ast^\HA( -)\), we get a factorization
  \[
  \F_2\{b_0,b_{2^{k+1}}\} \to \F_2[b_{2^{k+1}}] \to \F_2[{u}],
  \]
  where the composite sends \(b_0\) to \(1\) and \(b_{2^{k+1}}\) to \({u}^{2^k}\). Since the second map is a map of algebras, this identifies \(\pi_*(H\F_2 \smashover{\HA} \HA\aquot \alpha)\), as a comodule, with the subalgebra \(\F_2[{u}^{2^k}]\).
\end{proof}

Calculation of the relative Adams \(E_2\)-term then gives the following.

\begin{proposition}
\label{prop:changeofrings}
Let \(\alpha\) be as in Proposition~\ref{prop:HomologyAmodmodalpha}. The relative Adams spectral sequence computing the homotopy of \(\HA\aquot \alpha \smashover{\HA}M\) has a change-of-rings isomorphism
\[
E_2^{s,t}\cong \Ext_{\F_2[{u}]/{(u}^{2^k})}^{s,t}\big(\F_2, H_\ast^{\HA}(M)\big).
\]
\end{proposition}

\begin{proof}
The Hopf algebra \(\F_2[{u}^{2^k}]\) is cotensored up from \(\F_2\) along the quotient map 
\[
q:\F_2[{u}]\to \F_2[{u}]/({u}^{2^k}).
\]
In particular, for any comodule \(M\), we have an isomorphism of comodules
\[
\F_2[{u}^{2^k}]\otimes M\cong \F_2[{u}]\Boxover{\F_2[{u}]/{u}^{2^k}} q_\ast M,
\]
where \(q_\ast M\) is the ``restriction'' of \(M\) along the surjective map of Hopf algebras \(q\).
\end{proof}

\begin{corollary}
  Let \(\alpha\) be as in Proposition~\ref{prop:HomologyAmodmodalpha}.
  The relative Adams spectral sequence for \(\pi_*(\HA\aquot \alpha)\) degenerates at \(E_2\), with \(E_{\infty}\)-page given by
  \[
  \F_2[\xi_1,\dots,\xi_k].
  \]
\end{corollary}

This gives the \(\pi_\ast\HA\)-algebra structure up to associated graded. We do not know {\emph{a priori}} the images of \(\xi_{j}\) for \(j > k\); we know only that \(\xi_{j}\) is zero modulo elements of monomial degree at least \(2\), and this may depend on the exact terms in \(\alpha\). Identifying the relations imposed will require an operadic digression in Appendix~\ref{sec:CupkAlgs}.

\begin{definition}
  Let \({\HAn} = \HA \aquot \xi_{k+1}\) be the free associative \(\HA\)-algebra with \(\xi_{k+1} = 0\), and \(\barHAn = \HA \aquot \zeta_{k+1}\) be its conjugate.
\end{definition}

\begin{theorem}\label{thm:RingStructure}
As an \({\HA_{\ast}}\)-algebra, the coefficient ring \(\pi_* ({\HAn})\) is
  \[
    {\HA_{\ast}} / (\xi_{k+1}, \xi_{k+2}, \dots) \cong \F_2[\xi_1,\xi_2,\dots,\xi_k].
  \]
  Similarly, the coefficient ring \(\pi_*(\barHAn)\) is
  \[
    {\HA_{\ast}} / (\zeta_{k+1}, \zeta_{k+2}, \dots) \cong
    \F_2[\zeta_1,\zeta_2,\dots,\zeta_k].
  \]
\end{theorem}

\begin{proof}
  We will prove the statement for \(\barHAn = \HA \aquot \zeta_{k+1}\) first. The elements \(\zeta_1, \dots, \zeta_k\) of \({\HA_{\ast}}\) have the same images as \(\xi_1, \dots, \xi_k\) in the relative Adams spectral sequence, and so they determine polynomial generators of the \(E_\infty\) page of the relative Adams spectral sequence for \(\pi_\ast(\HA \aquot \alpha)\). In particular, there are no possible hidden multiplicative relations between them.
  
  Therefore, the composite
  \[
  \F_2[\zeta_1, \dots, \zeta_k] \subset {\HA_{\ast}} \to \pi_*(\barHAn)
  \]
  is an isomorphism, and so \({\HA_{\ast}} \to \pi_* \barHAn\) is surjective. This forces
  \[
      \pi_* (\barHAn) \cong {\HA_{\ast}} / J
  \]
  as an \({\HA_{\ast}}\)-algebra. We would like to determine this ideal \(J\); the element \(\zeta_{k+1}\) is in \(J\) by construction.
  By Corollary~\ref{cor:killingq1} in \S\ref{sec:cupktransgression}, the ideal \(J\) is closed under the Dyer--Lashof operation \(Q_1\). In \cite[\S III.2]{Hinfinity}, Steinberger showed that in the dual Steenrod algebra, the Dyer--Lashof operations satisfy 
  \[
  Q_1(\zeta_j) = \zeta_{j+1}.
  \]
  Therefore, the map \(\HA_\ast \to \HA_\ast/J\) factors through \({\HA_{\ast}}/(\zeta_{k+1}, \zeta_{k+2}, \dots)\) and this factorization is already forced to be an isomorphism.

  The same result holds for the \(\xi_k\) because conjugation is realized by a self-equivalence of \(H\F_2 \wedge H\F_2\) as a commutative ring spectrum.
\end{proof}
\begin{remark}
This can also be shown directly. If we apply \(Q_1\) to Milnor's formula \(\sum_{i+j=n} \xi_i^{2^j} \zeta_j = 0\) and use the Cartan formula \(Q_1(xy) = Q_1(x) y^2 + x^2 Q_1(y)\), it is possible to instead calculate directly that \(Q_1 \xi_j = \xi_{j+1} + \xi_1 \xi_j^2\). Since \(\xi_{k+1}\) maps to zero, all \(\xi_j\) for \(j > k\) must be in the ideal by induction.
\end{remark}

\begin{corollary}\label{cor:algquotientismodquotient}
  There is an equivalence of \(\HA\)-modules
  \[
    {\HAn} \simeq M_{> k}.
  \]
  Similarly, \(\barHAn \simeq \overline{M}_{>k}\).
\end{corollary}

\begin{proposition}\label{prop:HomologyAkm}
  As an \(\F_2[u]\)-comodule algebra, 
  \[
  H_\ast^{\HA}\big(\HAk\smashover{\HA} \barHAm\big)\cong\F_2[u^{2^k}]\otimes \F_2[u^{2^m}].
  \]
\end{proposition}
\begin{proof}
  The follows immediately from Propositions~\ref{prop:HomologyAmodmodalpha} and \ref{prop:MvsBarM}, as well as the K\"unneth theorem.
\end{proof}

\begin{corollary}\label{cor:RewrittenHomologyAkm}
  If \(m\geq k\), then we have an isomorphism of \(\F_2[u]\)-comodule algebras
  \[
    H_\ast^{\HA}\big(\HAk\smashover{\HA} \barHAm\big)\cong \F_2[u^{2^k}]\otimes \F_2[e_{2^{m+1}}],
  \]
  where \(e_{2^{m+1}}\) is a primitive element in degree \(2^{m+1}\).
\end{corollary}
\begin{proof}
  If \(m\geq k\), then both \(u^{2^m}\otimes 1\) and \(1\otimes u^{2^m}\) have the same coproduct. Their sum \(e_{2^{m+1}}\) is therefore primitive, and the two classes \(u^{2^k} \otimes 1\) and \(e_{2^{m+1}}\) are generators for the polynomial algebra.
\end{proof}

This gives the relative Adams \(E_2\)-term, via the change-of-rings isomorphism.

\begin{corollary}
  For \(m \geq k\), the relative Adams \(E_2\)-term for \(\HAk\smashover{\HA} \barHAm\) is
  \[
  \F_2[\xi_1,\dots,\xi_k]\otimes\F_2[e_{2^{m+1}}].
  \]
\end{corollary}

However, we need more algebraic information before we can analyze the relative Adams spectral sequence completely.

\section{Regularity in dual Steenrod algebra quotients}\label{sec:regsteenrodquotients}

In this section, we prove a series of algebraic results about the dual Steenrod algebra. Many of the inputs we need are standard facts in commutative algebra, and for convenience we will largely refer to Matsumura's text \cite{Matsumura}.

\begin{definition}
  Let \({\HAn_{\ast}}\) be the (graded) \({\HA_{\ast}}\)-algebra
  \[
  {\HA_{\ast}} / (\xi_{k+1}, \xi_{k+2}, \dots) \cong \F_2[\xi_1,\dots,\xi_k],
  \]
  let \(Z_{m+1,m+k}\) be the (graded) algebra
  \[
  \F_2[\zeta_{m+1},\zeta_{m+2},\dots,\zeta_{m+k}] \subset {\HA_{\ast}},
  \]
  and let
  \[
  \tRnast = \HAn_\ast / (\zeta_{k+1},\dots,\zeta_{2k}).
  \]
\end{definition}

In this section, our goal is to analyze the structures of the (graded) \({\HA_{\ast}}\)-algebras \(\HA_{\ast} / (\zeta_{m+1},\dots,\zeta_{m+k})\), such as \(\tRnast\). 

\begin{proposition}\label{prop:DeadZetas}
  In \(\HAk_{\ast}\), we have \(\zeta_n \in (\zeta_{m+1}, \zeta_{m+2}, \dots, \zeta_{m+k})\) for all \(n > m\).
\end{proposition}

\begin{proof}
  This is true by definition for \(m < n \leq m+k\). Milnor's formula for the conjugate classes states that
  \[
    0 = \sum_{i+j = n} \xi_i^{2^j} \zeta_j.
  \]
  In the ring \(\HAk_\ast\), this becomes the identity
  \[
    \zeta_n = \sum_{i=1}^k \xi_i^{2^{n-i}} \zeta_{n-i}
  \]
  whenever \(n > k\). If \(n > m+k\), then \(n-i > m\) for all \(1 \leq i \leq k\) that appear in the sum. By induction on \(n\), these elements \(\zeta_{n-i}\) are in the ideal, and so the result follows.
\end{proof}

Recall that \({\HA_{\ast}}\), as a Hopf algebra, represents the automorphisms of the additive formal group law over rings of characteristic \(2\): the formal power series
\[
f(x)=\sum_{i\geq 0}\xi_ix^{2^i},
\]
gives the universal automorphism, with composition inverse the power series
\[
g(x)=\sum_{i\geq 0}\zeta_i x^{2^i}.
\]
Here \(\xi_0 = \zeta_0 = 1\) by convention. 

\begin{proposition}
    Let \(I,J\subset\mathbb N\) be co-finite sets of natural numbers. Then the ring 
    \[
    B:=\HA_\ast/\big(\xi_i,\zeta_j\mid i\in I, j\in J\big)
    \]
    is a zero-dimensional ring and is finite-dimensional as an \(\F_2\)-vector space.
\end{proposition}
\begin{proof}
    Suppose that \(\mathfrak p\) is a (not necessarily graded) prime ideal of \(B\); this gives us a ring homomorphism
  \[
    q\colon \HA_\ast\to B/\mathfrak p,
  \]
  and hence automorphisms of the additive formal group law over \(B/\mathfrak p\)
  \[
    \alpha(x) = \sum_{i\geq 0} q(\xi_i)x^{2^i}\quad\text{and}\quad
    \beta(x) = \sum_{j\geq 0} q(\zeta_j)x^{2^j}
  \]
  which are composition-inverse to each other. By the co-finiteness assumption on \(I\) and \(J\), both of these automorphisms are polynomials in \(x\).
  
  However, the only composition-inverse polynomials with coefficients in an integral domain are linear, because the lead coefficient of the composite is a product of powers of leading coefficients. Therefore, for all \(i\), \(\xi_i = 0\) in \(B/\mathfrak p\), and hence \(\xi_i \in \mathfrak p\).

  This shows that the elements \(\xi_i\) are in every prime ideal, which means they are nilpotent. As \(B\) is generated by finitely many nilpotent generators, it is a finite-dimensional vector space. Moreover, the ideal \((\xi_1,\xi_2,\dots)\) is already maximal, so it is the only prime ideal of \(B\).
\end{proof}

\begin{corollary}\label{cor:finitedim}
  For any \(m \geq 0\), the algebra \({\HAk}_\ast/(\zeta_{m+1},\dots, \zeta_{m+k})\) is a finite-dimensional \(\F_2\)-vector space, and a zero-dimensional ring.
\end{corollary}

\begin{proof}
    By construction, we kill all but finitely many \(\xi_i\), and by Proposition~\ref{prop:DeadZetas}, only finitely many \(\zeta_j\) are non-zero modulo the ideal \((\zeta_{m+1}, \ldots, \zeta_{m+k})\).
\end{proof}

\begin{corollary}
  The ring \({\HAk_{\ast}}\) is a finitely-generated module over the ring \(Z_{m+1,m+k}\), with generators given by lifts of the basis of the quotient ring.
\end{corollary}

\begin{proposition}
  The composite ring homomorphism 
  \[
  Z_{m+1,m+k} \to {\HA_{\ast}} \to {\HAk_{\ast}}
  \]
  is one-to-one. 
\end{proposition}

\begin{proof}
  We have shown that the ring homomorphism \(g\co Z_{m+1,m+k} \to {\HAk_{\ast}}\) is finite. Such ring homomorphisms preserve proper inclusions of
  primes: if \(\mathfrak p \subsetneq \mathfrak q\) is a proper inclusion of prime ideals of
  \({\HAk_{\ast}}\), then the inclusion \(g^{-1}(\mathfrak p) \subset g^{-1}(\mathfrak q)\) of prime ideals of \(Z_{m+1,m+k}\) is also proper \cite[Theorem 9.3]{Matsumura}. 

  We then have a chain of prime ideals of \(Z_{m+1,m+k}\) of length \((k+2)\):
  \[
    0 \subset g^{-1}(0) \subsetneq g^{-1}(\xi_1) \subsetneq g^{-1}(\xi_1,\xi_2) \subsetneq  \dots \subsetneq
    g^{-1}(\xi_1,\dots,\xi_k).
  \]
  However, \(Z_{m+1,m+k}\) is a polynomial algebra in \(k\) variables over a field, and so it has Krull dimension \(k\) \cite[Theorem 5.6]{Matsumura}: the maximal number of distinct prime
  ideals in a chain in \(Z_{m+1,m+k}\) is \((k+1)\). Therefore, the first containment must be an equality.
\end{proof}

\begin{proposition}
  The ring \({\HAk_{\ast}}\) is a flat module over \(Z_{m+1,m+k}\).
\end{proposition}

\begin{proof}
  The ring \({\HAk_{\ast}}\) is a finitely-generated graded module over the graded subring \(Z_{m+1,m+k}\). The ring \({\HAk_{\ast}}\) is polynomial over a field, and hence Cohen--Macaulay \cite[Theorem 17.7]{Matsumura}.

  Let \(\mathfrak m = (\zeta_{m+1},\dots,\zeta_{m+k})\) be the augmentation ideal of the graded ring \(Z_{m+1,m+k}\); the quotient \(\tilde{\HA}\langle k \rangle_{\ast} = {\HAk_{\ast}}/\mathfrak m\) is zero-dimensional by Corollary~\ref{cor:finitedim}. The ring \({\HAk_{\ast}}\) is a flat module over \(Z_{m+1,m+k}\) if and only if the (ungraded) localization \(\big({\HAk_{\ast}}\big)_{\mathfrak m}\) is a flat module over the local ring \(\big(Z_{m+1,m+k}\big)_{\mathfrak m}\) \cite[\S 22, Remark]{Matsumura}. Applying the ``miracle flatness''
  theorem \cite[Theorem 23.1]{Matsumura} to the local homomorphism 
  \[
  \big(Z_{m+1,m+k}\big)_{\mathfrak m} \to 
  \big({\HAk_{\ast}}\big)_{\mathfrak m},
  \]
  we find that \({\HAk_{\ast}}\) is flat, and hence free, as a \(Z_{m+1,m+k}\)-module.
\end{proof}

\begin{theorem}\label{thm:RegularSequence}
  For any \(m \geq 0\), the elements \(\zeta_{m+1}, \dots, \zeta_{m+k}\) form a regular
  sequence in \({\HAk_{\ast}}\), and  \({\HAk_{\ast}}\) is a free module over \(Z_{m+1,m+k}\).
\end{theorem}

\begin{proof}
  These elements form a regular sequence in \(Z_{m+1,m+k}\), and this property is preserved by flat extensions.
\end{proof}

As a consequence of this regular sequence, we can determine the dimension function of these finite-dimensional quotients.

\begin{proposition}\label{prop:Poincare}
  The finite-dimensional graded ring \(\HAn_\ast / (\zeta_{m+1},\dots,\zeta_{m+k})\) has the following properties.
  \begin{enumerate}
  \item The Poincar\'e polynomial is given by the following rational function:
    \[
      \frac{\prod_{i=1}^k (1-t^{2^{m+i}-1})}{\prod_{i=1}^k
        (1-t^{2^i-1})}
    \]
  \item The dimension of \({\tRkast}\) as a vector space is the Gaussian
    binomial \(\binom{m+k}{k}_{q}\) evaluated at \(q=2\). 
  \item The maximal degree of a nonzero element is \(2(2^m-1)(2^k-1)\).
  \end{enumerate}
\end{proposition}

\begin{proof}
  The first formula is an immediate consequence of the fact that we have the quotient of a graded polynomial ring on generators in degrees \(2^i-1\) for \(1 \leq i \leq k\) by a regular sequence on generators in degrees \(2^{m+i}-1\) for \(1 \leq i \leq k\).

  The second formula is obtained by evaluating at \(t=1\), which we can calculate by either cancelling or via limit:
  \[
    \lim_{t \to 1} \frac{\prod_{i=1}^k (1-t^{2^{m+i}-1})}{\prod_{i=1}^k
        (1-t^{2^i-1})} = \frac{\prod_{i=1}^k (2^{m+i}-1)}{\prod_{i=1}^k
        (2^i-1)}
  \]
  This is the definition of the Gaussian binomial at \(q=2\).

  Finally, the degree of the Poincar\'e polynomial is the difference in degrees between numerator and denominator:
  \[
    \sum_{i=1}^k (2^{m+i}-1) - \sum_{i=1}^k(2^i - 1) = (2^{m+k+1} - 2^{m+1}) -
    (2^{k+1}-2) = 2(2^m-1)(2^k-1).\qedhere
  \]
\end{proof}

\begin{proposition}\label{prop:poincareduality}
  The ring \(\HAn_\ast / (\zeta_{m+1},\dots,\zeta_{m+k})\) is a (commutative, graded) \emph{Frobenius algebra} over \(\F_2\): 
  \begin{enumerate}
      \item if \(\delta=2(2^m-1)(2^k-1)\), then \(\HAn_\ast / (\zeta_{m+1},\dots,\zeta_{m+k})\) is a one-dimensional \(\mathbb F_2\)-vector space in degree \(\delta\), and
      \item for any \(t\), the multiplication between degrees \(t\) and \(\delta-t\) is a perfect pairing.
  \end{enumerate}
  As a result, it is injective as a module over itself.
\end{proposition}

\begin{proof}
  As the quotient of a polynomial ring by a regular sequence of length the Krull dimension, the ring \(\HAn_\ast / (\zeta_{m+1},\dots,\zeta_{m+k})\) is a zero-dimensional complete intersection ring. In particular, it is Gorenstein by \cite[Theorem 21.3]{Matsumura}, and by \cite[Proposition 7]{EisenbudGreenHarris} this means that the multiplication is a perfect pairing. By definition, this makes it a Frobenius algebra.
\end{proof}

We briefly recall the following result, re-expressing algebraic duality.

\begin{proposition}\label{prop:selfinjective}
  Suppose \(R\) is a graded Frobenius algebra over a field \(\F\), with top nonzero degree \(\delta\). For any graded left \(R\)-module \(M\), there is a natural isomorphism
  \[
  \Hom_R(M,R) \cong \Hom_\F(M, \Sigma^\delta\F).
  \]
  In particular, \(R\) is injective as a left module over itself. If \(R\) is commutative, then this natural isomorphism is one of \(R\)-modules.
\end{proposition}

\begin{proof}
  The multiplication pairing 
  \[
  \langle \mhyphen,\mhyphen\rangle\co R_t \otimes R_{\delta-t} \to \F_2
  \]
  is a perfect pairing that satisfies \(\langle xa,y\rangle = \langle x,ay\rangle\). As a result, \(y \mapsto \langle -,y\rangle\) is an isomorphism \(R \to \Hom_{\F_2}(R,\Sigma^\delta \F_2)\) of graded left \(R\)-modules. If \(R\) is commutative, we also have \(\langle ax, y\rangle = \langle x, ya\rangle\), extending it to an isomorphism of \(R\)-bimodules.

  Applying \(\Hom_R(M,-)\) gives the desired natural isomorphism of abelian groups, which is an isomorphism of \(R\)-modules in the commutative case (using the right module structure). Taking \(\F\)-linear duals is exact, so \(R\) is injective.
\end{proof}

\section{Finite \tpfstr{\(\HA\langle k\rangle \smashover{\HA}\overline{\HA\langle m\rangle}\)}{Akm}-algebras and splittings}\label{sec:splitting}
The graded algebra computation of \(\HAn_\ast\) actually greatly simplifies the study of \(\HAk\smashover{\HA}\barHAm\).
From Theorem~\ref{thm:RingStructure}, as an \(\HA\)-module, the ring \({\barHAm}\) is a colimit of modules of the form 
\[
\bigwedge_{j=m+1}^N \overline{M}_j, 
\]
and hence \({\HAk}\smashover{\HA} {\barHAm}\) is a colimit of left \({\HAk}\)-modules of the form 
\[
{\HAk}\smashover{\HA}\bigwedge_{j=m+1}^N \overline{M}_j.
\]

\begin{definition}
For natural numbers \(k,m,n \geq 0\), let
\[
    \Rkmn{m}{m+k+n}={\HAn}\smashover{\HA} \left(\bigwedge_{j=m+1}^{m+k+n} \overline{M}_j\right).
\]
For the special case \(k=m\) and \(n=0\), let
\[
\tRk = \HAk/(\zeta_{k+1},\ldots, \zeta_{2k}).
\]
\end{definition}

Theorem~\ref{thm:RegularSequence} shows that the collection of elements \(\zeta_{m+1},\dots,\zeta_{m+k}\) forms a regular sequence. In particular, the associated \(\Tor\) spectral sequence collapses.

\begin{corollary}\label{cor:TildeAkHomotopy}
  As an \({\HAk}_{\ast}\)-module, we have
  \[
  \pi_\ast\big(\HAn / (\zeta_{m+1},\dots,\zeta_{m+k})\big) \cong {\HAk}_\ast/(\zeta_{m+1},\dots,\zeta_{m+k}).
  \]
  In particular, \(\pi_\ast(\tRn) \cong \tRnast\).
\end{corollary}

Since \({\HAk}_\ast/(\zeta_{m+1},\dots,\zeta_{m+k})\) is a finite algebra, with its top degree class in \(2(2^m-1)(2^k-1)\), the \(\HA\)-module \(\Rkmn{m}{m+k}\) is also co-connective. In particular, there are no maps from more highly connected modules. This splits the quotients by later \(\xi\)s and \(\zeta\)s.

\begin{proposition}
  For each \(j>m+k\), we have an equivalence of \(\HA\)-modules:
  \[
    \Rkmn{m}{m+k}\smashover{\HA} M_j\simeq \Rkmn{m}{m+k}\vee \Sigma^{2^j}\Rkmn{m}{m+k}
  \]
  The same result also holds for \(\overline{M}_j\).
\end{proposition}
\begin{proof}
  By Proposition~\ref{prop:Poincare}, the degree of the top class in the Frobenius algebra \(\HAn_\ast/(\zeta_{m+1},\dots,\zeta_{m+k})\) is 
    \[
        2(2^m-1)(2^k-1)=2^{m+k+1}-2^{m+1} - 2^{k+1}+2.
    \] 
    If \(N\) is any \(\HA\)-module that is at least \(\big(2(2^m-1)(2^k-1)+1\big)\)-connective, then the space of maps
    \[
        \Map_{\HA}\big(N,\Rkmn{m}{m+k}\big)
    \]
    is contractible and all maps are null-homotopic as maps of \(\HA\)-modules. Applying this to the map
    \[
    \xi_j\co \Sigma^{2^j-1}\Rkmn{m}{m+k}\to  \Rkmn{m}{m+k},
    \]
    we see that since \(j\geq m+k+1\), the source of the multiplication-by-\(\xi_j\) may is sufficiently highly connected to guarantee that this map  is actually null-homotopic. The mapping cone \(\Rkmn{m}{m+k}\smashover{\HA}M_j\) then splits, and the result follows. The same proof applies to \(\overline{M}_j\).
\end{proof}

Inductively applying this result gives us the following splittings.
\begin{proposition}\label{prop:SmashSplitting}
  For any set \(I\) of natural numbers that are all greater than \(m+k\), there is a splitting of \(\HA\)-module spectra
  \[
  \Rkmn{m}{m+k} \smashover{\HA} M_I \simeq \Rkmn{m}{m+k} \wedge \left(\bigvee_{j \in D_I} S^j\right),
  \]
  where \(D_I\) is as in Notation~\ref{not:Dyadic}. The same result also holds for \(\overline M_I\).
\end{proposition}

\begin{corollary}\label{cor:HomotopySplitting}
  For any set \(I\) of natural numbers that are all greater than \(m+k\), the homotopy of \(\Rkmn{m}{m+k} \smashover{\HA} M_I\) is the \({\HAn}_{\ast}\)-module
  \[
  \HAn_\ast / (\zeta_{m+1},\dots,\zeta_{m+k}) \otimes \F_2\{e_j \mid j \in D_I\},
  \]
  where \(|e_j|=j\). The same result also holds for \(\overline M_I\).
\end{corollary}

These results can now be specialized to the set \(I = \{m+k+d \mid 1 \leq d \leq n\}\).

\begin{corollary}\label{cor:SplittingRkmn}
  For any \(n\geq 0\), including \(n=\infty\), we have a splitting 
  \[
  \Rkmn{m}{m+k+n} \simeq \Rkmn{m}{m+k}\wedge \left(\bigvee_{j=0}^{2^n-1} S^{j(2^{m+k+1})}\right)
  \]
  of modules over \(\HA\).
\end{corollary}

\begin{corollary}\label{cor:GapsInHomotopy}
  For each \(n\geq 0\), the homotopy of \(\Rkmn{m}{m+k+n}\) is the \({\HAn}_{\ast}\)-module
  \[
  \HAn_\ast / (\zeta_{m+1},\dots,\zeta_{m+k}) \otimes \left(\bigoplus_{j=0}^{2^n-1} \F_2\{e_{j(2^{m+k+1)}}\}\right).
  \]
\end{corollary}

These homotopy modules are all finite \(\F_2\)-vector spaces, and the top degree class is in dimension
\[
    2(2^m-1)(2^k-1)+(2^n-1)2^{m+k+1}=2^{m+k+n+1}-2^{m+1}-2^{k+1}+2.
\]

\begin{definition}
    For positive natural numbers \(k\) and \(m\) and natural number \(n\), let
    \[
    \delta=\delta(k,m,n)=2^{m+k+n+1}-2^{m+1}-2^{k+1}+2.
    \]
\end{definition}

Using this, we have a surprising consequence: these intermediate stages are actually all \(\HA\)-algebras.

\begin{theorem}\label{thm:SomeDoubleKilling}
  The \(\HA\)-module \(\Rkmn{m}{m+k+n}\) is equivalent to a Postnikov stage of \(\HAk\smashover{\HA}{\barHAm}\): there is an equivalence of \(\HA\)-modules
  \[
  \Rkmn{m}{m+k+n} \simeq P^{\delta(k,m,n)} 
  \Big( {\HAk}\smashover{\HA}{\barHAm}\Big).
  \]

  For any \(m,k\) and \(n\), \(\Rkmn{m}{m+k+n}\) has an associative algebra structure such that the unit 
  \[
  \HA\to \Rkmn{m}{m+k+n}
  \]
  is the natural map.
\end{theorem}
\begin{proof}
    Let \(\delta = \delta(k,m,n)\). Corollary~\ref{cor:SplittingRkmn} says that we have an equivalence of \(\HA\)-module spectra
    \[
        \HAk\smashover{\HA}\barHAm\simeq \Rkmn{m}{m+k+n}\vee\bigvee_{r\geq 1}\Sigma^{r \cdot 2^{m+k+n+1}}\Rkmn{m}{m+k+n}.
    \]
    The first summand is \((\delta+1)\)-co-connective while the second is at least \((\delta+1)\)-connective, so we deduce that the inclusion
    \[
    \Rkmn{m}{m+k+n}\hookrightarrow \HAk\smashover{\HA}\barHAm
    \]
    becomes an equivalence on \(\delta\)'th Postnikov sections.

    Since \(\HA\) is connective, there is a Postnikov tower in \(\HA\)-modules that preserves associative algebras, meaning that all Postnikov truncations are natural associative \(\HA\)-algebras and all truncation maps are associative algebra maps, as shown by Dugger--Shipley \cite{DuggerShipleyPostnikov}. This gives us an associative algebra structure on 
  \[
  P^{\delta}\big(\HAk\smashover{\HA}\barHAm\big)\simeq \Rkmn{m}{m+k+n},
  \]
  with a canonical map of algebras from \(\HAk\smashover{\HA}\barHAm\).
\end{proof}

\begin{corollary}
  There is an equivalence of \(\HA\)-algebras
  \[
  \Rkmn{m}{m+k+n} \simeq \barHAm/(\xi_{k+1},\dots,\xi_{k+m+n}).
  \]
\end{corollary}

\begin{proof}
  These two algebras are equivalent Postnikov stages of \(\HAn \smashover{\HA} \barHAm\).
\end{proof}

\begin{remark}
Because of this equivalence, we can focus almost exclusively moving forward on the case where \(m\geq k\).
\end{remark}
  
We can also describe the relative homology of these ring spectra as comodules using the K\"unneth isomorphism; the following shows that the algebra structure is forced.

\begin{proposition}\label{prop:HomologyOfRkn}
    For any natural numbers \(k, m, n\geq 0\), we have an isomorphism of \(\F_2[{u}]\)-comodule algebras
    \[
    H_\ast^{\HA}\big(\Rkmn{m}{m+k+n}\big)\cong
    \F_2[{u}^{2^k}]\otimes \F_2[u^{2^m}]/ \left(u^{2^{k+m+n}} \otimes 1 + 1 \otimes u^{2^{k+m+n}}\right)
    \]
\end{proposition}
\begin{proof}
Theorem~\ref{thm:SomeDoubleKilling} shows that the composite map of \(\HA\)-modules
\[
\HAn/(\zeta_{m+1}, \dots, \zeta_{m+k+n}) \to \HAn \smashover{\HA} \barHAm \to P^\delta \big(\HAn \smashover{\HA} \barHAm\big),
\]
is an equivalence, and the latter map is a map of algebras. Applying the K\"unneth isomorphism gives us a composite
\[
\F_2[u^{2^k}] \otimes H_\ast^\HA(M_{m+1,\dots,m+k+n}) \to \F_2[u^{2^k}] \otimes \F_2[u^{2^m}] \to H_\ast^\HA\Big(P^\delta\big(\HAn \smashover{\HA} \barHAm\big)\Big),
\]
where the first map is a map of comodules and the second is a map of comodule algebras.
In particular, the second map is surjective, and by Corollary~\ref{cor:HomologyOfMI} the image of the first is 
\[
\F_2[u^{2^k}] \otimes \{1,u^{2^m},\dots,u^{2^{k+m+n}-1}\}.
\]
The kernel must be generated by a primitive element with leading term \(1 \otimes u^{2^{k+m+n}}\), but there is only one such element. We deduce that the Postnikov truncation induces a quotient map of rings
\[
\F_2[{u}^{2^k}]\otimes \F_2[u^{2^m}]\to \F_2[{u}^{2^k}]\otimes \F_2[u^{2^m}]/ \left(u^{2^{k+m+n}} \otimes 1 + 1 \otimes u^{2^{k+m+n}}\right)
\]
as desired.
\end{proof}

\begin{corollary}\label{cor:homologyofRkmn}
  For any natural numbers \(m \geq k\) and \(n \geq 0\), we have an isomorphism of comodule algebras
  \[
  H_\ast^\HA\big(\Rkmn{m}{m+k+n}\big) \cong \F_2[u^{2^k}] \otimes \F_2[e_{2^{m+1}}]/(e_{2^{m+1}}^{2^{k+n}})
  \]
  where \(e_{2^{m+1}}\) is the primitive of Corollary~\ref{cor:RewrittenHomologyAkm}.
\end{corollary}

\begin{corollary}
    When \(m \geq k\), the map on relative homology induced by the ring map
    \[
        \Rkmn{m}{m+k+n}\to \HAk/(\zeta_{m},\dots,\zeta_{m+k+n})
    \]
    is the ring map specified by
    \[
        e_{2^{m+1}}\mapsto e_{2^m}^2.
    \]
\end{corollary}

\begin{remark}
This calculation rules out any way to make the map
\[
\Rkmn{m}{m+k+n} \to  \HAk\smashover{\HA}\barHAm
\]
into a map of algebras that is compatible with the Postnikov identification, because there is no section of the map of algebras on relative homology.
\end{remark}

By Corollary~\ref{cor:homologyofRkmn}, we also have a nice finitary result for relative Adams spectral sequences.

\begin{corollary}\label{cor:AdamsETwoRkn}
For any natural numbers \(m\geq k\) and \(n\geq 0\), the relative Adams spectral sequence computing the homotopy of \(\Rkmn{m}{m+k+n}\) has
\[
E_2^{s,t}\cong \F_2[\xi_1,\dots,\xi_k]\otimes \F_2[e_{2^{m+1}}]/(e_{2^{m+1}}^{2^{n+k}}).
\]
\end{corollary}

For \(n\geq 0\), all of the relative Adams spectral sequences carry essentially the same information, due to the \(\HA\)-module splitting of these into \(\Rkmn{m}{m+k}\)-modules from Corollary~\ref{cor:SplittingRkmn}.

\begin{proposition}
  For any \(n\geq 1\), the class 
  \[
  e_{2^{m+k+1}}=e_{2^{m+1}}^{2^k}
  \]
  is a permanent cycle in the relative Adams spectral sequence for \(\Rkmn{m}{m+k}\).
\end{proposition}
\begin{proof}
  The splitting 
  \[
  \Rkmn{m}{m+k+n}\simeq \bigvee_{j=0}^{2^n-1} \Sigma^{j(2^{m+k+1})} \Rkmn{m}{m+k}
  \]
  shows that we have a non-trivial homotopy class \(\tilde{e}_{2^{m+k+1}}\) in degree \(2^{m+k+1}\) which has a non-trivial Hurewicz image. The class \(e_{2^{m+k+1}}\) is the only primitive in homology in that degree, and must therefore be a permanent cycle.
\end{proof}

\begin{corollary}\label{cor:HomotopyRkn}
  We have an isomorphism of algebras for any \(n\geq 0\):
  \[
  \pi_\ast\big(\Rkmn{m}{m+k+n}\big)\cong \HAk_\ast/(\zeta_{m+1},\dots,\zeta_{m+k})\otimes \F_2[\tilde{e}_{2^{m+k+1}}]/\tilde{e}_{2^{m+k+1}}^{2^n}.
  \]
\end{corollary}

In fact, the topological splitting shows an even stronger result.
\begin{proposition}\label{prop:SplittingAdams}
  For any \(\HA\)-module \(M\), the relative Adams spectral sequence for the homotopy of \(\Rkmn{m}{m+k+n}\smashover{\HA}M\)
  is a direct sum of suspensions of copies of the relative Adams spectral sequence for \(\Rkmn{m}{m+k}\smashover{\HA}M\): for all \(r\geq 2\),
  \begin{multline*}
  E_r\big(\Rkmn{m}{m+k+n}\smashover{\HA}M\big)\cong \\ E_r\big(\Rkmn{m}{m+k}\smashover{\HA}M\big)\otimes\F_2[e_{2^{m+k+1}}]/e_{2^{m+k+1}}^{2^n},
  \end{multline*}
  and \(e_{2^{m+k+1}}\) is a permanent cycle.
\end{proposition}
\begin{proof}
  This follows from smashing the splitting of \(\Rkmn{m}{m+k+n}\) with \(M\).
\end{proof}

\section{Duality}\label{sec:duality}

Recall from Proposition~\ref{prop:poincareduality} that the graded ring
\[
\F_2[\xi_1,\dots,\xi_k]/(\zeta_{m+1},\dots,\zeta_{m+k})\cong \pi_\ast \left(\Rkmn{m}{m+k}\right)
\]
is a commutative graded Frobenius algebra. We can promote such algebraic duality on the coefficient ring of \(\Rkmn{m}{m+k}\) to a spectral duality, including Gorenstein duality.

\begin{theorem}\label{thm:finiteduality}
  Suppose that \(R\) is a connective associative ring spectrum over \(H\F\) whose coefficient ring is a commutative graded Frobenius algebra over a field \(\F\) with top-dimensional class in degree \(\delta\). Then, for \(R\)-modules \(M\), there is a natural isomorphism
  \[
  \pi_* D_R M \cong  \left(\pi_{\delta-\ast}(M)\right)^\vee
  \]
  as modules over \(\pi_\ast R\), where \((-)^\vee\) denotes the \(\F\)-vector space dual.
\end{theorem}

\begin{proof}
  Because \(\pi_* R\) is self-injective, the universal coefficient spectral sequence
  \[
  \Ext_{\pi_* R}^{s,t} (\pi_* M, \pi_* R) \Rightarrow \pi_{t-s} D_R M
  \]
  always degenerates. This gives a natural isomorphism
  \[
  \pi_* D_R M \cong \Hom_{\pi_* R}(\pi_* M, \pi_* R).
  \]
  (In other words, Spanier--Whitehead duality is a type of Brown--Comenetz duality in the category of \(R\)-modules.) Expanding the identification \(\pi_* R \cong \Hom_{\F}(\pi_* R, \Sigma^\delta \F)\) from Proposition~\ref{prop:selfinjective}, this becomes an isomorphism 
  \[
  \Hom_{\pi_*R}(\pi_* M, \pi_*R) \cong \Hom_{\F}(\pi_* M, \Sigma^\delta \F)
  \]
  as desired.
\end{proof}

This gives us spectral Gorenstein duality as in Dwyer--Greenlees--Iyengar \cite{DwyerGreenleesIyengar}.

\begin{corollary}\label{cor:dualityAk}
  A ring spectrum \(R\) as in Theorem~\ref{thm:finiteduality} satisfies a Gorenstein duality condition of shift \(\delta\): we have an equivalence of \(R\)-module spectra
  \[
  F_{R}\big(H\F, R\big) \simeq \Sigma^{\delta} H\F.
  \]
\end{corollary}

Spectral duality for such \(R\) then gives us duality for its modules. 
\begin{proposition}\label{prop:dualityformodules}
Suppose that \(R\) is an \(\HA\)-algebra whose coefficient ring is a commutative graded Frobenius algebra over \(\F_2\), with top-dimensional class in degree \(\delta\).
\begin{enumerate}
\item \label{itemIJ}
For \(I,J\subseteq \N\) finite subsets, with \(s(I,J) = \sum\limits_{i\in I} 2^i + \sum\limits_{j\in J} 2^j\), there is an isomorphism
\[\pi_{*}\left(R\smashover{\HA} M_{I}\smashover{\HA} \overline{M}_J\right)
  \cong 
\left(  \pi_{\delta+s(I,J)-*}\left(R \smashover{\HA}
M_{I}\smashover{\HA}\overline{M}_J\right)\right)^\vee\]
of \(\pi_*R\)-modules.
\item \label{itemleqr} There is an isomorphism \[\pi_{*}\left(R\smashover{\HA} \End_{\HA}(M_{\leq r})\right) \cong \left( \pi_{\delta-*}\left(R\smashover{\HA} \End_{\HA}(M_{\leq r})\right)\right)^\vee \]
of \(\pi_*R\)-modules.
\end{enumerate}
Here \((-)^\vee\) denotes the \(\F_2\)-vector space dual.
\end{proposition}

\begin{remark}
  By Proposition~\ref{prop:poincareduality} and Corollary~\ref{cor:HomotopyRkn}, the algebras
  \[
  \pi_\ast(\Rkmn{m}{m+k+n})
  \]
  are Frobenius algebras, and so Proposition~\ref{prop:dualityformodules} holds for all of the \(\HA\)-algebras \(\Rkmn{m}{m+k+n}\), with shift factor \(\delta(k,m,n)\).
\end{remark}

\begin{proof}
Let \(D_A(M):=F_{\HA}(M,\HA)\) be the Spanier--Whitehead duality functor in the category of \(\HA\)-modules.
Then
\[ D_{\HA}(M_{I}\smashover{\HA} \overline{M}_J)\simeq \Sigma^{-s(I,J)} M_{I}\smashover{\HA} \overline{M}_J.\]
Therefore, the same is true after base change to \(R\), i.e.,
\[ D_{R}(R\smashover{\HA}M_{I} \smashover{\HA} \overline{M}_J)\simeq \Sigma^{-s(I,J)}R\smashover{\HA} M_{I}\smashover{\HA} \overline{M}_J.\]
Combining this with Theorem~\ref{thm:finiteduality}, we get:
\begin{align*} 
\pi_{*+s(I,J)}\left(R\smashover{\HA} M_{I}\smashover{\HA} \overline{M}_J \right)
&\cong \pi_*D_{R}\big(R \smashover{\HA}
  M_{I}\smashover{\HA} \overline{M}_J\big) \\
  &\cong \left(\pi_{\delta-*}\left(R \smashover{\HA}
M_{I}\smashover{\HA}\overline{M}_J\right)\right)^\vee
\end{align*}
This proves item~\eqref{itemIJ}. For \eqref{itemleqr}, the proof is completely analogous, using the fact that \(\End_{\HA}(M_{\leq r})\) is Spanier--Whitehead self-dual because it is the endomorphism ring of a dualizable object:
\[
D_\HA(\End_{\HA}(M_{\leq r})) \simeq D_{\HA}\big(M_{\leq r}\smashover{\HA} D_{\HA}(M_{\leq r})\big)\simeq \End_{\HA}(M_{\leq r}). \qedhere
\]
\end{proof}

\section{Computations of relative Adams spectral sequences}\label{sec:Examples}

The homotopy rings of the ring spectra \({\HAk}\) can immediately be read out of the computations, since they are just the quotients by collections of generators. For the ring spectra \(\tRn\) (or more generally \(\Rkmn{m}{m+k+n}\)), the homotopy rings were determined in Corollary~\ref{cor:HomotopyRkn}, but this formulation can be hard to parse due to the complicated formul{\ae} for the classes \(\zeta_i\) in terms of the \(\xi_i\). The relative Adams spectral sequence can be used here to describe what happens up to associated graded.

For more general quotients, the relative Adams spectral sequence is a very helpful approach. In this case, we can simplify the procedure by finding the smallest quotient algebra \({\HAk}\) of \(\HA\) over which a given object is a module. The relative Adams spectral sequence here is a spectral sequence of modules, giving the desired result. Prototypical examples of this are given by the associative algebras \(\End_A(M_{<k})\), especially when base changed to \(\tRk\). Here, the differentials we describe control almost everything.

\subsection{Some families of differentials}\label{sec:ExampleRkn}

The relative Adams spectral sequence for modules over \(\HAn\) are the most important by Proposition~\ref{prop:SplittingAdams}. We begin by describing some basic differentials. For any \(n>k\), Corollary~\ref{cor:MkRelHomology} shows that the relative Adams spectral sequence for \(\HAk \smashover{\HA} \overline M_n\) has \(E_2\)-term
\[
\F_2[\xi_1,\dots,\xi_k]\otimes \F_2\{b_0, b_{2^{n}}\}.
\]

\begin{proposition}\label{prop:BasicAdamsDifferentials}
For any \(n=kq+r\) with \(q\geq 1\) and \(1\leq r \leq k\), in the relative Adams spectral sequence for \(\HAk \smashover{\HA} \overline M_n\) we have a differential
\[
d_{1+2^r+2^{r+k}+\dots+2^{n-k}}(b_{2^{n}})=\xi_r\xi_k^{2^r+2^{r+k}+\dots+2^{n-k}}.
\]
The next page is \(E_\infty\):
\[
\F_2[\xi_1,\dots,\xi_k]/(\xi_r\xi_k^{2^r+2^{r+k}+\dots+2^{n-k}}).
\]
\end{proposition}
\begin{proof}
    Since we have a cofiber sequence
    \[
        \Sigma^{2^{n}-1}\HAk\xrightarrow{\zeta_{n}} \HAk\to \HAk \smashover{\HA} \overline M_n,
    \]
    the long exact sequence in homotopy gives us that the homotopy of the quotient is
    \[
    \pi_\ast \big(\HAk \smashover{\HA} \overline M_n\big)\cong \F_2[\xi_1,\dots,\xi_k]/(\zeta_{n}).
    \]
    The image of \(\zeta_n\) in the relative Adams spectral sequence must then be hit by a differential. This image records only the leading term of \(\zeta_{n}\) in terms of monomial degree. By Lemma~\ref{lem:ZetaModM} below, this is the stated term.
\end{proof}

\begin{lemma}\label{lem:ZetaModM}
    For any \(n=kq+r\) with \(q \geq 1\) and \(1\leq r \leq k\), we have
    \[
    \zeta_n\equiv \xi_r\xi_k^{2^r+2^{r+k}+\dots+2^{n-k}}\mod (\xi_{k+1},\xi_{k+2},\dots)+\text{higher degree monomials}.
    \]
\end{lemma}
\begin{proof}
    Let \(j_q=2^r+2^{r+k}+\dots+2^{n-k}\), and let \(\mathfrak m\) denote the augmentation ideal of \(\F_2[\xi_1,\dots,\xi_k]\). We will verify the formula by induction on \(q\), working modulo terms of degree greater than \(j\). Since we are working modulo \(\xi_{k+1}\) and higher, Milnor's formula gives
    \[
        \zeta_n=\sum_{i=1}^{k} \zeta_{n-i}\xi_i^{2^{n-i}}\equiv \zeta_{n-k}\xi_k^{2^{n-k}}\mod\mathfrak m^{j_q+2}.
    \]
    If \(q=1\), then \(\zeta_{n-k}=\zeta_r\), which gives the base case. If \(q>1\), then by induction, 
    \[
    \zeta_{n-k}\equiv \xi_r\xi_k^{2^{j_{m-1}}}\mod \mathfrak m^{j_{q-1}+2},
    \]
    so 
    \[
        \zeta_n\equiv\zeta_{n-k}\xi_k^{2^{n-k}}\equiv \zeta_r\xi_k^{2^{j_{q-1}}+2^{n-k}}\mod \mathfrak m^{j_{q-1}+2^{n-k}+2}
    \]
    as needed.
\end{proof}

\begin{remark}
    It is only in identifying the leading term of \(\zeta_{n}\) that we used the assumption that \(q\geq 1\). For \(m= 0\), the homology changes, since \(H^{\HA}_\ast (M_{n})\) no longer restricts to a trivial \(\F_2[u^{2^n}]\)-comodule. 
\end{remark}

We can apply this in the more general case of the quotient by all \(\xi_{m+i}\) for \(1\leq i\leq k\). Recall from Corollary~\ref{cor:AdamsETwoRkn} that the relative Adams \(E_2\) term for \(\HAk/(\zeta_{m+1},\dots,\zeta_{m+k})\), for \(m\geq k\), is
\[
\F_2[\xi_1,\dots,\xi_k]\otimes \F_2[e_{2^{m+1}}]/e_{2^{m+1}}^{2^{k}}.
\]

\begin{theorem}\label{thm:BasicAdamsDiffs}
    For each \(1\leq i\leq k\), write \(m+i=kq+r\), with \(1\leq r\leq k\). In the relative Adams spectral sequence for \(\HAk/(\zeta_{m+1},\dots,\zeta_{m+k})\),  we have
    \[
    d_s\big(e_{2^{m+1}}^{2^{i-1}}\big)=\mycases{ 
    0 & s<1+2^r+2^{r+k}+\dots+2^{m+i-k} \\
    \xi_{i}\xi_k^{2^r+2^{r+k}+\dots+2^{m+i-k}} & s=1+2^r+2^{r+k}+\dots+2^{m+i-k}.
    }
    \]
\end{theorem}
\begin{proof}
    On relative homology, the map
    \[
    \HAk \smashover{\HA} \overline{M}_{m+i} \to \HAk/(\zeta_{m+1},\dots,\zeta_{m+k})
    \]
    takes the class \(b_{2^{m+i}}\) to the class \(e_{2^{m+1}}^{2^{i-1}}\). The result follows from naturality of the relative Adams spectral sequence.
\end{proof}

\begin{corollary}
    In the relative Adams spectral sequence for \(\tRn\), for each \(1\leq i\leq k\), we have
    \[
    d_r\big((e_{2^{k+1}})^{2^{i-1}}\big)=\mycases{ 
    0 & r<2^i+1 \\
    \xi_{i}\xi_k^{2^i} & r=2^i+1.
    }
    \]
\end{corollary}

More difficult is understanding the remaining differentials. The reason is that, while the classes \(\zeta_{k+1},\dots,\zeta_{2k}\) form a regular sequence in \(\HAn_\ast\), the leading terms do not. This forces some higher differentials. We now spell this out for \(\RTwoFour\). More surprising, however, is that in various quotients of \(\tRn\) wherein we kill also {\emph{lower}} \(\xi\)s and \(\zeta\)s, these differentials control the entire relative Adams spectral sequence. We show examples of this in Sections~\ref{sec:AllButXiTwo} and \ref{sec:AllButXiThree}.

\subsection{The relative Adams spectral sequence for \tpfstr{\(\tilde \HA\langle 2\rangle\)}{A~<2>}}\label{sec:Atilde2}

\begin{figure}[ht]
    \centering
    \includegraphics[width=\textwidth]{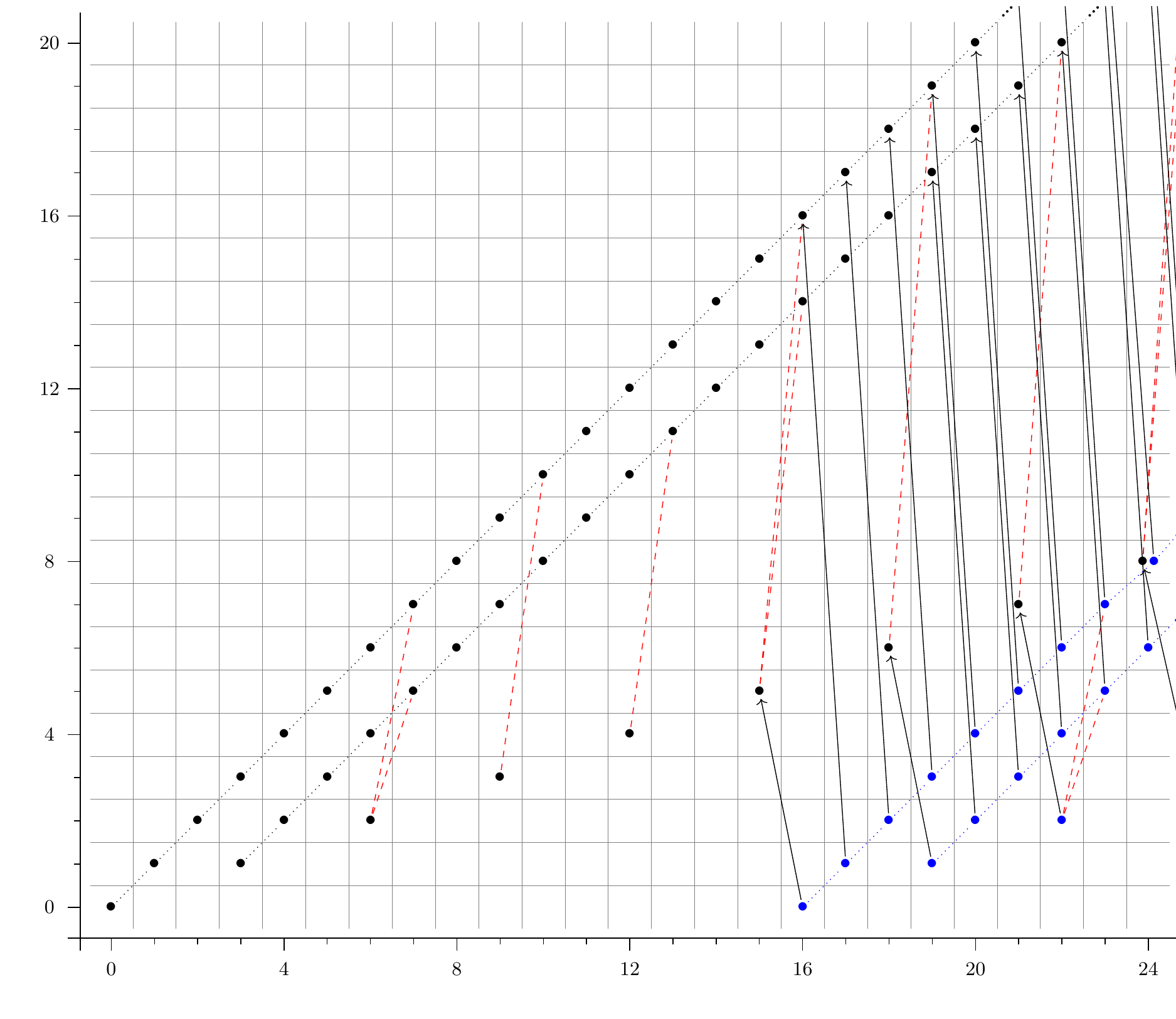}
    \caption{The relative Adams spectral sequence starting at \(E_5\) for \(\RTwoFour\)}
    \label{fig:ASSRTwoFour}
\end{figure}

As a comodule algebra, Proposition~\ref{prop:HomologyOfRkn} shows we have
\[
H_\ast^{\HA}\big(\RTwoFour\big)\cong \F_2[{u}^4]\otimes \F_2[e_8]/e_8^4,
\]
where \(e_8\) is a primitive degree \(8\) class. The relative Adams \(E_2\)-term is therefore
\[
E_2=\F_2[\xi_1,\xi_2]\otimes \F_2[e_8]/e_8^4,
\]
with \(e_8\) in relative Adams filtration \(0\). Figure~\ref{fig:ASSRTwoFour} shows the spectral sequence from \(E_5\) on; we encourage the reader to use this as a guide to the computation. In this, bullets indicate copies of \(\F_2\). The red dashed lines are ``exotic'' \(\xi_1\)-multiplications. The classes in blue are the submodule generated by the class \(e_8^2\).

Proposition~\ref{prop:BasicAdamsDifferentials} (or Corollary~\ref{cor:AdamsdThree}) gives our first differential, killing \(\xi_3 + \zeta_3\).

\begin{proposition}
    There is a \(d_3\)-differential
    \[
        d_3(e_8)=\xi_1\xi_2^2.
    \]
\end{proposition}

This gives an \(E_4\) page
\[
\F_2[\xi_1,\xi_2]/(\xi_1\xi_2^2)\otimes E(e_8^2).
\]
Proposition~\ref{prop:BasicAdamsDifferentials} then gives the differential on \(e_8^{2}\), killing \(\xi_4 + \zeta_4\).

\begin{proposition}
    There is a \(d_5\)-differential
    \[
        d_5(e_8^2)=\xi_2^5.
    \]
\end{proposition}
Note that this leaves \(\xi_1^k e_8^2\) for \(k\geq 1\) as a \(d_5\)-cycle, since
\[
d_5(\xi_1^k e_8^2)=\xi_1^k d_5(e_8^2)=\xi_1^k \xi_2^5=d_3(e_8)\xi_1^{k-1}\xi_2^3=0,
\]
and we also have \(\xi_1^k\xi_2 e_8^2\) for \(k\geq 1\) for similar reasons.

\begin{proposition}
    As an \(\F_2[\xi_1,\xi_2]\)-module, \(E_6\) is
    \[
    \F_2[\xi_1,\xi_2]/(\xi_1\xi_2^2,\xi_2^5)\oplus \F_2[\xi_1,\xi_2]/(\xi_2^2)\cdot \{\xi_1e_8^2\}.
    \]
\end{proposition}

We now remark that, in the \emph{abutment} \(\F_2[\xi_1,\xi_2]/(\zeta_3, \zeta_4)\) of this spectral sequence, there is an unexpected relation:
\[
\xi_1^{16} = \xi_1 \zeta_4 + (\xi_2^3 + \xi_1^6 \xi_2 + \xi_1^9) \zeta_3 \equiv 0
\]
Therefore, \(\xi_1^{16}\) must be the target of a differential. The only class available to accomplish this is \(\xi_1 e_8^2\).
We deduce the following. 

\begin{proposition}
    We have \(d_{15}\)-differentials
    \[
        d_{15}(\xi_1^k e_8^2)=\xi_1^{15+k}\text{ and }d_{15}(\xi_1^k\xi_2e_8^2)=\xi_1^{15+k}\xi_2,
    \]
    for \(k\geq 1\).
\end{proposition}

This leaves an \(E_{16}\)-page
\[
\F_2[\xi_1,\xi_2]/(\xi_1\xi_2^2, \xi_2^5, \xi_1^{16}).
\]
All generators here are permanent cycles, so the spectral sequence degenerates at this point.

\begin{remark}
The differentials correspond to relations that are not homogeneous with respect to the degree filtration, imposing  multiplicative extensions on the spectral sequence. To solve them, we must carefully remember to expand {\emph{all}} formul{\ae} modulo the corresponding jump in filtration. For example, trying to work with \(\zeta_4\) modulo only filtration at least \(6\) gives an incorrect pattern of differentials including \(d_{13}\) and \(d_{17}\). Careful calculation of the differentials is equivalent to calculation of the leading terms in a Gr\"obner basis for the ideal \((\zeta_{k+1},\dots, \zeta_{2k})\).
\end{remark}

\begin{remark}
The relations imposed by setting \(\zeta_i = 0\) are complicated enough to make determining the exact structure of the finite-dimensional algebras \(\tRnast\) complex. Calculations with Sage indicate some conjectural patterns:
\begin{enumerate}
    \item for \(j < k\), \(2^{2k-j+1}\) is the smallest power of \(\xi_j\) which is zero in \(\tRnast\), and
    \item the element \(\xi_k^{2^{k+1}-2}\) generates the topmost nonzero degree in \(\tRnast\).
\end{enumerate}
\end{remark}

\subsection{The homotopy of \tpfstr{\(\tilde \HA\langle 2\rangle \wedge_{\HA}\End_{\HA}(M_1)\)}{EndM smash A~<2>}}
\label{sec:AllButXiTwo}

In this section, we will discuss the coefficient ring of 
\[
\RTwoFour\smashover{\HA}\End_{\HA}(M_1)\simeq \Sigma^{-2} \RTwoFour\smashover{\HA} M_{1}\smashover{\HA} \overline{M}_{1}.
\]
Since this is an associative algebra under \(\RTwoFour\), the relative Adams spectral sequence is a spectral sequence of modules over that of \(\RTwoFour\). We present a picture through a range with all differentials and extensions in Figure~\ref{fig:ASSAllButXiTwo}.

Our descriptions of the homologies and relative Adams \(E_2\)-terms for \(\HAn\) and for \(\End_{\HA}(M_{\leq k})\) give the following.

\begin{corollary}
As an associative algebra over the relative Adams \(E_2\)-term for \(\RTwoFour\), the relative Adams \(E_2\)-term for \(\pi_\ast (\RTwoFour\smashover{\HA}\End_{\HA}(M_1))\) is
\[
\F_2[\xi_1, \xi_2]\otimes E(\beta_{\Neg2}) / (\xi_1 \beta_{\Neg2}, \xi_1+ \xi_2 \beta_{\Neg2}  ) \otimes \F_2[e_8]/(e_8^4),
\] 
with \(\beta_{\Neg2}, e_8\) in relative Adams filtration \(0\).
\end{corollary}

\begin{figure}[ht]
    \centering
    \includegraphics[width=\textwidth]{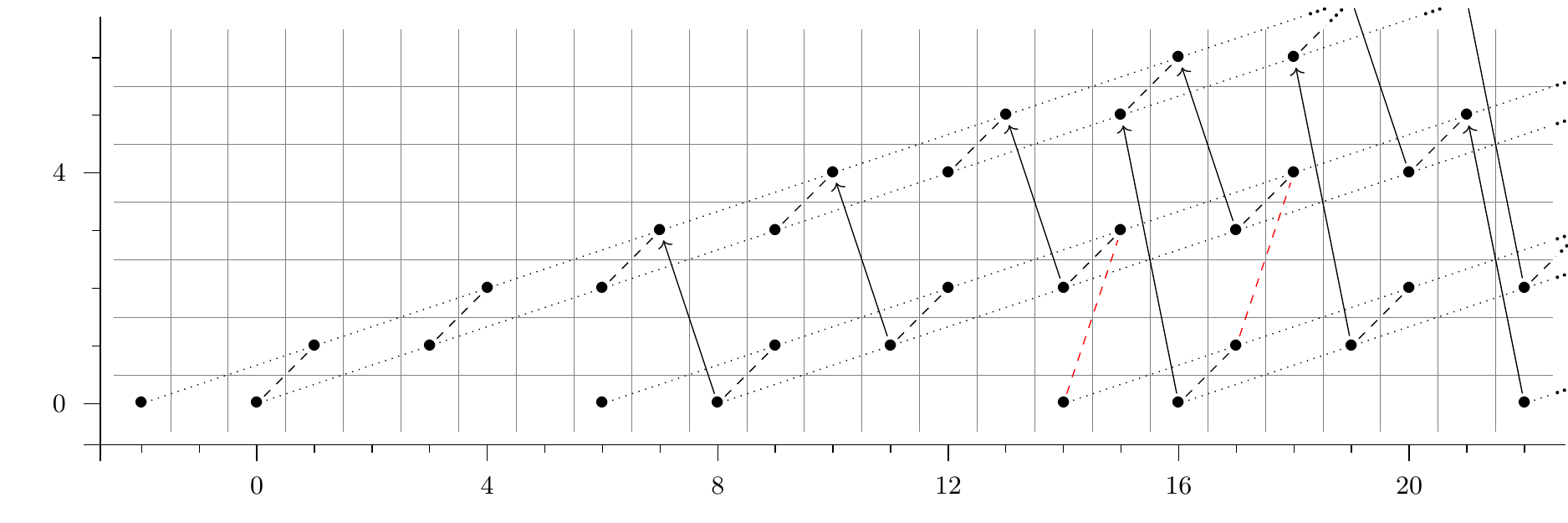}
    \caption{The relative Adams spectral sequence for \(\RTwoFour\smashover{\HA}\End_{\HA}(M_{1})\)}
    \label{fig:ASSAllButXiTwo}
\end{figure}

The differential \(d_3(e_8) = \xi_1 \xi_2^2\) for \(\RTwoFour\) carries through directly, using the relation
\[
\xi_1=\xi_2\beta_{\Neg2}.
\] 

\begin{proposition}\label{prop:d3Everythingbutxi2}
There are \(d_3\)-differentials
\[
d_3(e_8) = \xi_2^3 \beta_{\Neg2}\text{ and }
d_3(e_8^3) = \xi_2^3 \beta_{\Neg2} e_8^2.
\]
\end{proposition}

Similarly, the differential \(d_5(e_8^2) = \xi_2^5\) carries over without change.

\begin{proposition}\label{prop:d5DiffAllButXiTwo}
  There are \(d_5\)-differentials
  \[
    d_5(e_8^2) = \xi_2^5\text{ and }
    d_5(e_8^3\beta_{\Neg2}) = \xi_2^5 e_8\beta_{\Neg2}.
  \]
\end{proposition}

\begin{proposition}
There are hidden \(\xi_1\)-multiplications
\begin{align*}
  \xi_1 \cdot e_8^2 \beta_{\Neg2} &= \xi_2^3 e_8 \beta_{\Neg2}\\  
  \xi_1 \cdot \xi_2e_8^2 \beta_{\Neg2} &= \xi_2^4 e_8 \beta_{\Neg2}. 
\end{align*}
\end{proposition}

\begin{proof}
This follows from the duality of Section~\ref{sec:duality}. We apply Proposition~\ref{prop:dualityformodules}, item~\eqref{itemleqr} in the case \(r=1\) and \(k=2\). We find that the homotopy groups of this spectrum are, as a module, self-dual with a shift of \(18\). The \(\xi_1\)-multiplication out from degrees \(0\) and \(3\) then resolves the \(\xi_1\)-multiplication into degrees \(18\) and \(15\).
\end{proof}

\subsection{The homotopy of \tpfstr{\(\tilde \HA\langle 3\rangle \wedge_{\HA}\End_{\HA}(M_{\leq 2})\)}{EndM smash A~<3>}}\label{sec:AllButXiThree}
In this section, we will discuss the coefficient ring of 
\[
\RThreeSix\smashover{\HA}\End_{\HA}(M_{\leq 2})\simeq \Sigma^{-6} \RThreeSix\smashover{\HA} M_{\leq 2}\smashover{\HA} \overline{M}_{\leq 2}.
\]
This ring spectrum is an associative algebra over the spectrum \(\RThreeSix\) and its relative Adams spectral sequence is an algebra over the relative Adams spectral sequence of \(\RThreeSix\). 

\begin{proposition}
The relative Adams \(E_2\)-term for \(\pi_*\big(\RThreeSix\smashover{\HA}\End_{\HA}(M_{\leq 2})\big)\) is:
\[
\F_2[\xi_1,\xi_2,\xi_3]\otimes E(\beta_{\Neg2},\beta_{\Neg4}) /
\big(\xi_1+\xi_2\beta_{\Neg2}, \xi_2+\xi_3\beta_{\Neg4}\big)\otimes\F_2[e_{16}]/e_{16}^8.
\]
\end{proposition}

The differentials in Theorem~\ref{thm:BasicAdamsDiffs} will imply all the differentials in the relative Adams spectral sequence for \(\pi_*\big(\RThreeSix\smashover{\HA}\End(M_{\leq 2})\big)\). All the differentials are shown in Figure~\ref{fig:ASSAllButXi3Diff}, and the \(E_\infty\)-page, together with the exotic multiplications by \(\xi_1\) and \(\xi_2\) are shown in Figure~\ref{fig:ASSAllButXi3EinftyDiff}.

\begin{proposition}
For \(0\leq k\leq 3\), there are \(d_3\)-differentials
\[
d_3(e_{16}^{2k+1}) = \xi_3^3 e_{16}^{2k} \beta_{\Neg2}\beta_{\Neg4}.
\]
\end{proposition}
\begin{proof}
Since \(e_{16}^2\) is a \(d_3\)-cycle, it suffices to show the case \(k=0\). There, we have by naturality
\[
d_3(e_{16}) = \xi_3^2 \xi_1
= \xi_3^2 \xi_2 \beta_{\Neg2}
= \xi_3^3 \beta_{\Neg2}\beta_{\Neg4}.\qedhere
\]
\end{proof}

Note that this leaves \(d_3\)-cycles 
\[
e_{16}^{2k+1}\beta_{\Neg2}, e_{16}^{2k+1}\beta_{\Neg4},\text{ and }e_{16}^{2k+1}\beta_{\Neg2}\beta_{\Neg4}.
\]

\begin{proposition}
For all \(\epsilon_1=0,1\) and \(\epsilon_2=0,1\), we have \(d_5\)-differentials
\[
d_5\big(e_{16}^{2+4\epsilon_1}(e_{16}\beta_{\Neg2})^{\epsilon_2}\big) = \xi_3^5 \beta_{\Neg4} e_{16}^{4\epsilon_1}(e_{16}\beta_{\Neg2})^{\epsilon_2}.
\]
\end{proposition}
\begin{proof}
Naturality shows that we have the differential
\[
d_5(e_{16}^2) = \xi_3^4 \xi_2= \xi_3^5 \beta_{\Neg4}.
\]
The remaining differentials follow from the ring structure.
\end{proof}

With the ring structure and the relation imposed by the \(d_3\)-differential, we see that the classes
\[
e_{16}^{2+4\epsilon_1}\beta_{\Neg2}\quad \text{and} \quad e_{16}^{2+4\epsilon_1}\beta_{\Neg4}
\]
are \(d_5\)-cycles.

\begin{proposition}
There are \(d_9\)-differentials
\begin{align*}
d_9(e_{16}^{4}) &=  \xi_3^9 \\
d_9(e_{16}^{4}\beta_{\Neg2}) &=  \xi_3^9 \beta_{\Neg2} \\
d_9(e_{16}^{5}\beta_{\Neg2}) &=  \xi_3^9 e_{16} \beta_{\Neg2}\\
d_9(e_{16}^{5}\beta_{\Neg4}) &=  \xi_3^9 e_{16}\beta_{\Neg4} \\
d_9(e_{16}^{6}\beta_{\Neg2}) &=  \xi_3^9 e_{16}^2 \beta_{\Neg2}\\
d_9(e_{16}^{6}\beta_{\Neg4}) &=  \xi_3^9 e_{16}^2 \beta_{\Neg4}\\
d_9(e_{16}^{7}\beta_{\Neg4}) &=  \xi_3^9e_{16}^3 \beta_{\Neg4} \\
d_9(e_{16}^{7}\beta_{\Neg2}\beta_{\Neg4}) &=  \xi_3^9 e_{16}^3 \beta_{\Neg2}\beta_{\Neg4}.
\end{align*}
\end{proposition}
\begin{proof}
 These differentials all follow from applying the Leibniz rule to the \(d_9\)-differential 
 \[
 d_9(e_{16}^4) = \xi_3^9\] 
 in the relative Adams spectral sequence for \(\pi_* (\HA \langle 3 \rangle/(\xi_4, \xi_5, \xi_6))\). 
\end{proof}

\begin{proposition}\label{prop:xi1ExtensionAllButxi3}
  There is a hidden \(\xi_1\)-multiplication
  \[
    \xi_1 \cdot e_{16}^2\beta_{\Neg2}  = \xi_3^3 e_{16}\beta_{\Neg2}\beta_{\Neg4} 
  \]
  into degree 31.
\end{proposition}

\begin{proof}
The previous results show that we have a \(d_3\)-differential
\[
d_3(e_{16}) = \xi_3^3 \beta_{\Neg2} \beta_{\Neg4}.
\]
The surviving class \(e_{16}\beta_{\Neg2} \beta_{\Neg4} \) then detects the Massey product \(\langle \beta_{\Neg2} \beta_{\Neg4}, \beta_{\Neg2}, \xi_3^3 \beta_{\Neg4}\rangle\). There is also a \(d_5\)-differential
\[
d_5(e_{16}^2) = \xi_3^5 \beta_{\Neg2},
\]
and the surviving class \(e_{16}^2\beta_{\Neg2} \) then detects the Massey product \(\langle \beta_{\Neg2}, \xi_3^3 \beta_{\Neg4}, \xi_3^2\rangle\). Multiplying by \(\xi_1 = \xi_3 \beta_{\Neg2} \beta_{\Neg4}\), we get a juggling formula:
\[
\beta_{\Neg2} \beta_{\Neg4} \langle \beta_{\Neg2}, \xi_3^3 \beta_{\Neg4}, \xi_3^2\rangle = \xi_3^2 \langle \beta_{\Neg2} \beta_{\Neg4}, \beta_{\Neg2}, \xi_3^3 \beta_{\Neg4} \rangle
\]
Multiplying through by \(\xi_3\), we find that \(\xi_1 (e_{16}^2\beta_{\Neg2}) = \xi_3^3 ( e_{16}\beta_{\Neg2} \beta_{\Neg4})\) up to indeterminacy. The indeterminacy consists of multiples of \(\xi_3^2 \beta_{\Neg2} \beta_{\Neg4}\), of which there are none in this degree.
\end{proof}

After multiplication with \(\xi_3\), the extension in Proposition~\ref{prop:xi1ExtensionAllButxi3} also implies that there is a hidden \(\xi_1\)-multiplication
\[
    \xi_1 \cdot \xi_3 e_{16}^2\beta_{\Neg2}  = \xi_3^4 e_{16}\beta_{\Neg2}\beta_{\Neg4} 
  \]
  into degree 38.
  
\begin{proposition}
We have the following hidden \(\xi_1\)-multiplications: 
\begin{align*}
\xi_1 \cdot e_{16}^4 \beta_{\Neg4} &= \xi_3^7 e_{16} \beta_{\Neg4} \,\,\,(\text{into degree } 61)\\
\xi_1 \cdot \xi_3 e_{16}^4 \beta_{\Neg4} &= \xi_3^8 e_{16} \beta_{\Neg4} \,\,\,(\text{into degree } 68)\\
\xi_1 \cdot e_{16}^6 \beta_{\Neg2}\beta_{\Neg4} &= \xi_3^7 e_{16}^3 \beta_{\Neg2}\beta_{\Neg4} \,\,\,(\text{into degree } 91)\\
\xi_1 \cdot \xi_3e_{16}^6 \beta_{\Neg2}\beta_{\Neg4} &= \xi_3^8 e_{16}^3 \beta_{\Neg2}\beta_{\Neg4} \,\,\,(\text{into degree } 98)
\end{align*}
and the following hidden \(\xi_2\)-multiplications:
\begin{align*}
\xi_2 \cdot e_{16}^4 \beta_{\Neg2}\beta_{\Neg4} &=\xi_3^7 e_{16} \beta_{\Neg4} \,\,\,(\text{into degree } 61) \\
\xi_2 \cdot \xi_3 e_{16}^4 \beta_{\Neg2}\beta_{\Neg4} &=\xi_3^8 e_{16} \beta_{\Neg4} \,\,\,(\text{into degree } 68) \\
\xi_2 \cdot e_{16}^4 \beta_{\Neg4} &= \xi_3^{5}e_{16}^2 \beta_{\Neg4} \,\,\,(\text{into degree } 63) \\
\xi_2 \cdot \xi_3 e_{16}^4 \beta_{\Neg4} &= \xi_3^{6}e_{16}^2 \beta_{\Neg4} \,\,\,(\text{into degree } 70) \\
\xi_2 \cdot \xi_3^2e_{16}^4 \beta_{\Neg4} &= \xi_3^{7}e_{16}^2 \beta_{\Neg4} \,\,\,(\text{into degree } 77) \\
\xi_2 \cdot \xi_3^3 e_{16}^4 \beta_{\Neg4} &= \xi_3^{8}e_{16}^2 \beta_{\Neg4} \,\,\,(\text{into degree } 84) \\
\xi_2 \cdot e_{16}^5 \beta_{\Neg2}\beta_{\Neg4}&= \xi_3^5 e_{16}^3 \beta_{\Neg2} \beta_{\Neg4} \,\,\,(\text{into degree } 77)\\ 
\xi_2 \cdot \xi_3 e_{16}^5 \beta_{\Neg2}\beta_{\Neg4}&= \xi_3^6 e_{16}^3 \beta_{\Neg2} \beta_{\Neg4} \,\,\,(\text{into degree } 84)\\ 
\xi_2 \cdot \xi_3^2 e_{16}^5 \beta_{\Neg2}\beta_{\Neg4}&= \xi_3^7 e_{16}^3 \beta_{\Neg2} \beta_{\Neg4} \,\,\,(\text{into degree } 91)\\ 
\xi_2 \cdot \xi_3^3e_{16}^5 \beta_{\Neg2}\beta_{\Neg4}&= \xi_3^8 e_{16}^3 \beta_{\Neg2} \beta_{\Neg4} \,\,\,(\text{into degree } 98)\\ 
\xi_2 \cdot e_{16}^6 \beta_{\Neg2}\beta_{\Neg4} &= \xi_3^7e_{16}^3\beta_{\Neg4} \,\,\,(\text{into degree } 93) \\ 
\xi_2 \cdot \xi_3 e_{16}^6 \beta_{\Neg2}\beta_{\Neg4} &= \xi_3^8e_{16}^3 \beta_{\Neg4} \,\,\,(\text{into degree } 100).
\end{align*}
These extensions are shown in Figure~\ref{fig:ASSAllButXi3EinftyDiff}. 
\end{proposition}
\begin{proof}
By Proposition~\ref{prop:dualityformodules}, item~\eqref{itemleqr} with \(k=3\) and \(r=2\)
we find that the homotopy groups of this spectrum are, as a module, self-dual with a shift of 98. By duality, all the extensions now follow from known extensions (Proposition~\ref{prop:xi1ExtensionAllButxi3}) and product structures on the \(E_2\)-page. 
\end{proof}

\afterpage{\begin{landscape}
\begin{figure}[p]
    \includegraphics[width =8in]{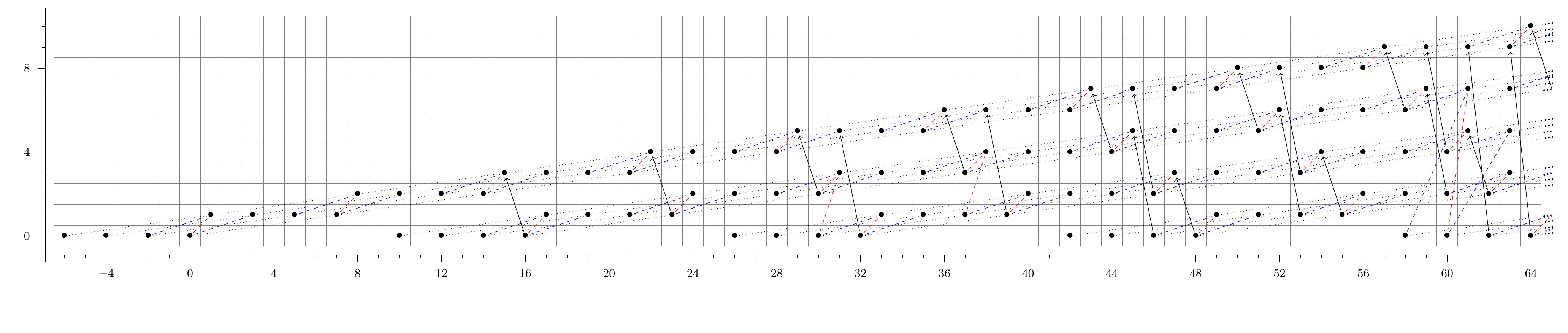}
    \includegraphics[width =8in]{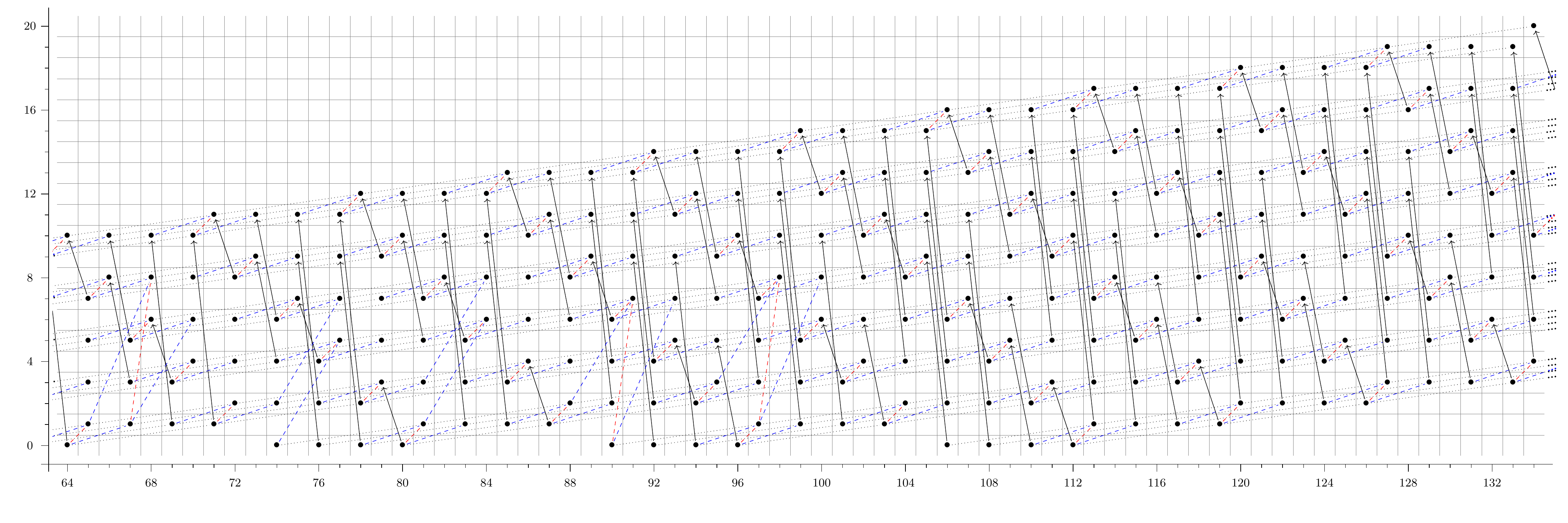}
    \caption{The relative Adams spectral sequence for \(\pi_*\big(\RThreeSix\smashover{\HA}\End_{\HA}(M_{\leq 2})\big)\)}
    \label{fig:ASSAllButXi3Diff}
\end{figure}

\begin{figure}[p]
    \includegraphics[width =8in]{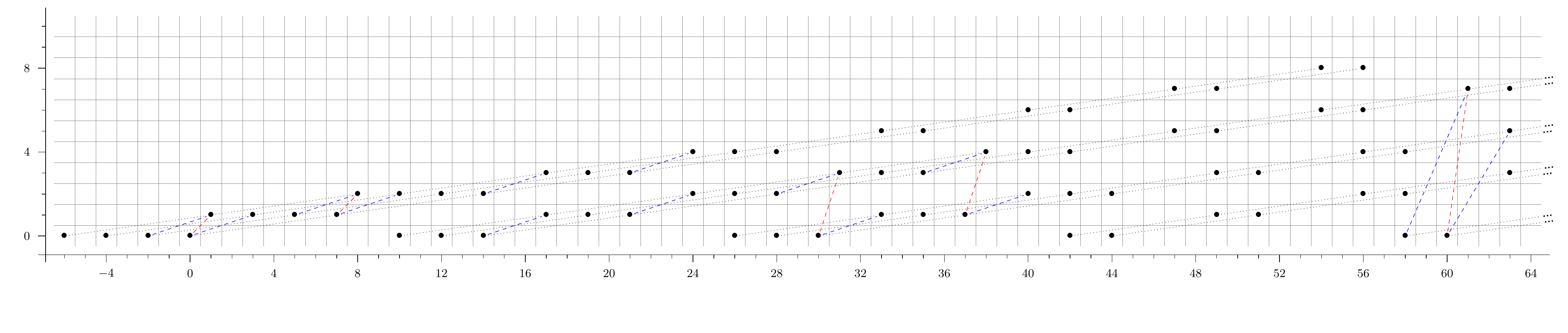}
    \includegraphics[width =8in]{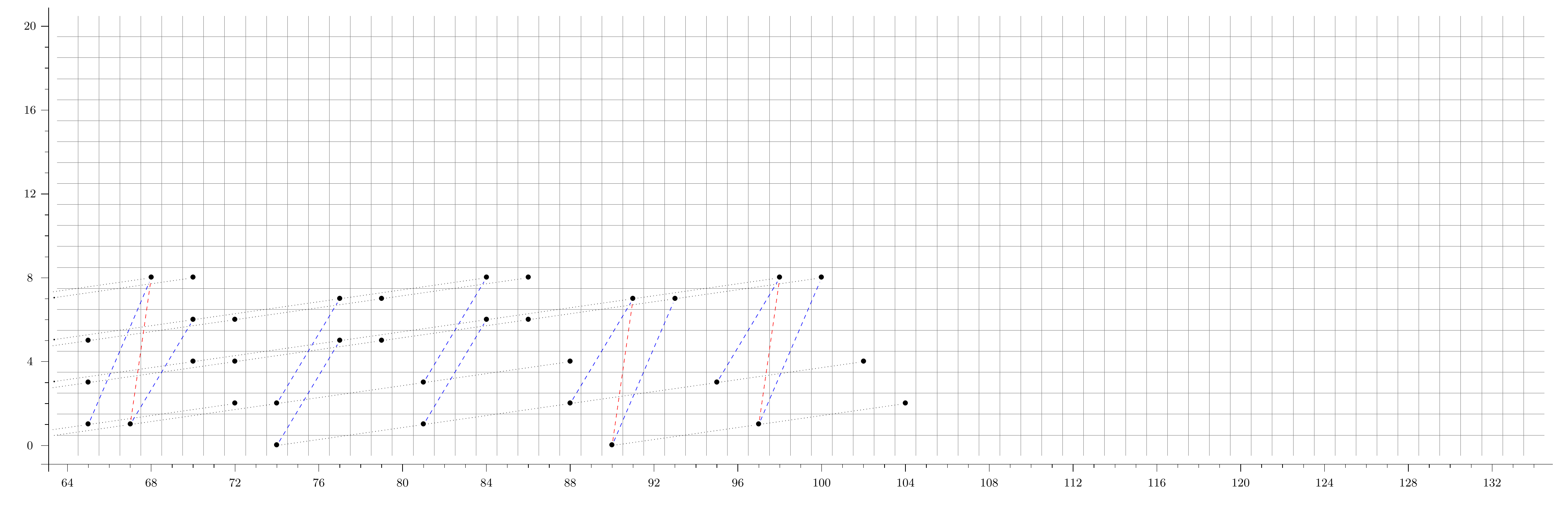}
    \caption{The \(E_\infty\)-page of the relative Adams spectral sequence for \(\pi_*\big(\RThreeSix\smashover{\HA}\End_{\HA}(M_{\leq 2})\big)\), with hidden extensions}
    \label{fig:ASSAllButXi3EinftyDiff}
\end{figure}
\end{landscape}}

\appendix

\section{Multiplicative operations}\label{sec:CupkAlgs}

Important to our calculations of the quotient rings \(\HA \aquot \alpha\) is the knowledge that the Dyer--Lashof operation \(Q_1\) preserves the kernel of the map \({\HA_{\ast}} \to \pi_*(\HA \aquot \alpha)\). To show this, we need a brief operadic digression.

The main precedent for this is the transgression. Serre's transgression theorems assert the compatibility of maximal-length differentials in the Serre spectral sequence with the Steenrod squaring operations in \(H^\ast(X)\) \cite{SerreSS}, and May showed similar results for the Dyer--Lashof operations in \(H_\ast(\Omega^k X)\) \cite{MayAlgebraic}. On the chain level, a cycle \(x \in C\) has an extended power \(Q_k(x)\) represented by Steenrod's cup-\(k\) product \(x \smile_k x\). If \(f\co C \to D\) is a map of algebras with these operations, and \(f(x) = \partial y\), then
\[
f(Q_k(x)) = \partial \left[f(x) \smile_k y + y \smile_{k-1} y\right].
\]
The formula on the right only depends on a cup-\((k-1)\) product in \(D\) and a cup-\(k\) product expressing that the image of \(f\) is, in some sense, central in \(D\). In this section, we will show that this is part of a more general story.

\subsection{Cup-\tpfstr{\(k\)}{k} algebras}\label{sec:cupkfirst}

\begin{notation}
    For \(k \geq 1\), let \(\Sksig[k]\) denote the unit sphere in the \(k\)-dimensional real sign representation of \(C_2\): a \((k-1)\)-sphere with the antipodal \(C_2\)-action. Let  \(D(k\sigma)\) be the corresponding closed disk.
\end{notation}

\begin{definition}
  Suppose that \(R\) is a commutative ring spectrum.  A \emph{cup-\(k\) algebra} over \(R\) is an \(R\)-module spectrum \(A\) with a map 
  \[
  m_k\colon \Sksig[(k+1)]_+ \smashover{C_2} (A \smashover{R} A) \to A.
  \]
  A \emph{map} of cup-\(k\) algebras is a map preserving this structure.
\end{definition}

\begin{remark}
  Equivalently, a cup-\(k\) algebra is a \(C_2\)-equivariant map
  \[
  \Sksig[(k+1)] \to \Map_R(A \smashover{R} A, A).
  \]
  Any \(E_{k+1}\)-operad \(\mathcal{O}\) has a \(C_2\)-equivariant homotopy equivalence  
  \[\Sksig[(k+1)] \simeq \mathcal{O}(2).\]
  Forgetting all but the binary operations then determines a forgetful functor from \(E_{k+1}\) algebras to cup-\(k\) algebras.
\end{remark}

\begin{example}
  A cup-zero algebra is an \(R\)-module \(A\) with a binary multiplication map \(m\co A \smashover{R} A \to A\).
\end{example}

More generally, we can get an inductive description of cup-\(k\) algebras.

\begin{definition}
  Suppose that \(A\) is a cup-\(k\) algebra. The \emph{bracket} is the composite map
    \[
    \Sigma^k (A \smashover{R} A) \simeq \big(C_{2+}\wedge S^{k-1}\big)\smashover{C_2} (A\wedge A) \to \Sksig[(k+1)]_+ \smashover{C_2} A \smashover{R} A \to A.
    \]
\end{definition}

\begin{remark}
On homotopy groups, the bracket induces a \(\pi_*R\)-bilinear (up to sign) \emph{Browder bracket}
    \[
      [-,-]\co \pi_m(A) \times \pi_n(A) \to \pi_{m+k+n}(A).
    \]
\end{remark}
\begin{proposition}
  A cup-\(k\) algebra is equivalent to a cup-\((k-1)\) algebra with a chosen nullhomotopy of the bracket \(\Sigma^{k-1} A \smashover{R} A \to A\).
\end{proposition}

\begin{proof}
We have a natural equivariant inclusion \(\Sksig[k]\hookrightarrow S\big((k+1)\sigma\big)\), and the quotient is the pointed space \(C_{2+}\wedge S^{k}\). Stably, this determines a cofiber sequence
\[
C_{2+}\wedge\Sigma^\infty S^{k-1}\to \Sigma^\infty_+ \Sksig[k]\to \Sigma^{\infty}_+ S\big((k+1)\sigma).
\]
Smashing over \(C_2\) with \(A \smashover{R} A\) gives us a cofiber sequence of \(R\)-modules, and the result follows by definition of cup-\(k\) structures.
\end{proof}

\begin{example}
  A cup-one algebra structure is equivalent to a choice of a binary multiplication \(m\co A \smashover{R} A \to A\), and a chosen ``commutativity homotopy'' \(h\co m \Rightarrow m \circ \tau\), where \(\tau\) is the twist map.
\end{example}

\subsection{Operations}\label{sec:cupkops}

Analogously to the work of Bruner
\cite[\S IV.7]{Hinfinity}, a cup-\(k\) algebra has power operations.

\begin{proposition}
  Suppose that \(A\) is a cup-\(k\) algebra and that \(\alpha\co S^n \to A\) is a map. Then there is an induced natural map
  \[
    Q(\alpha)\co R \wedge \Sigma^n \RP_n^{n+k} \to A
  \]
  from a shifted stunted projective space to \(A\), whose restriction to the bottom cell \(S^{2n}\) is the square of \(\alpha\).
\end{proposition}

\begin{proof}
  The \(C_2\)-equivariant map \(\alpha \wedge \alpha\co S^n \wedge S^n \to A\wedge A\) determines a composite map of \(R\)-modules
  \[
    R \wedge \left(\Sksig[(k+1)] \smashover{C_2} (S^n \wedge S^n)\right) \to \Sksig[(k+1)] \smashover{C_2} A \smashover{R} A \to A.
  \]
  The equivariant smash product \(\Sksig[(k+1)]_+ \smashover{C_2} (S^n \wedge S^n)\) appearing on the left is the Thom spectrum of the equivariant bundle \((n + n\sigma)\) on \(\Sksig[(k+1)]/C_2 = \RP^k\), and this Thom spectrum is a shifted stunted projective space. (See \cite[Theorem V.2.14]{Hinfinity}.)

  By definition, the multiplication \(m\co A \smashover{R} A \to A\) is the composite
  \[
  A \smashover{R} A \cong C_{2+} \smashover{C_2} (A \smashover{R} A) \to \Sksig[(k+1)]_+ \smashover{C_2} (A \smashover{R} A) \to A.
  \]
  If we compose with the map \(\alpha \wedge \alpha\co S^n \wedge S^n \to A \smashover{R} A\), the result is the square of \(\alpha\).
\end{proof}

\begin{corollary}
  If \(n\) is even or \(2=0\) in \(\pi_0(R)\), there exist natural operations (sometimes called cup-\(1\) squares)
  \[
    Q_1\co \pi_{n}(A) \to \pi_{2n+1}(A)
  \]
  for cup-one \(R\)-algebras.
\end{corollary}

\begin{proof}
  The stunted projective space \(\RP_n^{n+1}\) is the cofiber of the attaching map
  \[
  \big(1 - (-1)^n\big)\co S^n \to S^n
  \]
  for the \((n+1)\)-cell. If \(n\) is even or \(2=0\) in \(\pi_0(R)\), this implies that there is a splitting
  \[
    R \wedge \RP_n^{n+1} \simeq \Sigma^n R \vee \Sigma^{n+1} R.
  \]
  Choosing such a splitting in the homotopy category, any map \(\alpha\co S^n \to A\) determines a natural composite operation
  \[
    S^{2n+1} \to \Sigma^n R \wedge \RP_n^{n+1}
    \xrightarrow{Q(\alpha)} A.\qedhere
  \]
\end{proof}

\begin{example}
  When \(R\) is \(H\F_2\), \(Q_1\) is the Dyer-Lashof operation of the same name.
\end{example}

\begin{remark}
  Two different splittings may determine two different operations \(Q_1\) and \(Q'_1\). However, the difference involves the square: by considering the homotopy groups of \(\Sigma^n R \vee \Sigma^{n+1} R\), there exists some element \(u \in \pi_1 R\) such that
  \[
  Q_1(\alpha) - Q'_1(\alpha) = u \alpha^2
  \]
  for all \(\alpha\). This means that \(Q_1(\alpha)\) may involve choices, but the ideal generated by \(\alpha\) and \(Q_1(\alpha)\) does not.
\end{remark}

\subsection{Centrality}\label{sec:cupkcentrality}
\begin{definition}
  Suppose that \(A\) is a cup-\(k\) algebra. A \emph{central cup-\((k-1)\) algebra over \(A\)} consists of the following data:
  \begin{enumerate}
      \item a cup-\((k-1)\) algebra \(B\),
      \item a map \(\eta\co A \to B\) of cup-\((k-1)\) algebras, and
      \item a factorization
      \[
      \begin{tikzcd}
      \left(\Sksig[k]_+ \wedge A \smashover{R} B\right) \coprod_{\Sksig[k]_+ \wedge A \smashover{R} A}
      \left(D(k\sigma)_+ \wedge A \smashover{R} A\right) \ar[d] \ar[dr] &\\
      D(k\sigma)_+ \wedge A \smashover{R} B \ar[r,dotted]
      & B.
      \end{tikzcd}
      \]
      Here the diagonal map from the pushout is induced by the commutative diagram
      \[
        \begin{tikzcd}[row sep = small]
        \Sksig[k]_+ \wedge A \smashover{R} A \ar[rr] \ar[dd] &&
        D(k\sigma)_+ \wedge A \smashover{R} A \ar[d]\\
        && A \ar[d] \\
        \Sksig[k]_+ \wedge A \smashover{R} B \ar[r] &
        \Sksig[k]_+ \wedge B \smashover{R} B \ar[r] &
        B,
        \end{tikzcd}
      \]
      with the upper-right vertical arrow being the map that trivializes the bracket operation of \(A\). 
  \end{enumerate}
\end{definition}

\begin{remark}
  This last commutative square helps to clarify the definition. The map \(D(k\sigma)_+ \wedge A \smashover{R} A \to A\) encodes the trivialization of the bracket on \(A\). Centrality asks for a \emph{compatible} trivialization of the bracket \(\Sigma^{k-1} A \smashover{R} B \to B\), and so the Browder bracket \([\eta(a), b]\) always vanishes.
  One aspect of this compatibility is that the two trivializations of \([\eta(a),\eta(a')] = \eta[a,a']\) are equivalent.
\end{remark}

\begin{remark}Our definition of centrality for these cup products captures the degree-\(2\) part of a more complicated structure. Namely, suppose that \(A\) is an \(E_k\) algebra. The tensor product over \(A\) gives the category \(LMod_A\) of left \(A\)-modules an \(E_{k-1}\)-monoidal structure. If \(B\) is an \(E_{k-1}\)-algebra in \(LMod_A\), then \(B\) is a central cup-\((k-1)\) algebra over \(A\).\end{remark}

\begin{example}
  Suppose that \(B\) is an \(E_{k-1}\) \(R\)-algebra. The object \(R\) is a strict unit for the tensor product of \(R\)-modules, which gives it a cup-\(k\) algebra structure via the projection
  \[
  \Sksig[(k+1)]_+ \smashover{C_2} R \smashover{R} R \to R.
  \]
  The fact that \(R\) is a strict unit for the \(E_{k-1}\)-multiplication on \(B\) makes \(B\) a central cup-\((k-1)\) algebra over \(R\).
\end{example}

Unwinding the adjoint maps in the definition of a central cup-\((k-1)\) algebra, we obtain the following alternative definition of the structure in terms of mapping spaces.
\begin{proposition}
  A central cup-\(k\)-algebra structure on \(B\) is equivalent to a commutative diagram of spaces:
  \[
  \begin{tikzcd}[row sep=small, column sep=small]
  \Sksig[k] \ar[dr] \ar[drr] \ar[dd] \\
  &D(k\sigma) \ar[r]
  & \Map_R(A \smashover{R} A, A) \ar[dd]\\
  \Map_R(B \smashover{R} B, B) \ar[dr] \ar[drr]\\
  &\Map_R(A \smashover{R} B, B) \ar[r] \ar[uu,leftarrow, crossing over]
  &\Map_R(A \smashover{R} A, B)
  \end{tikzcd}
  \]
\end{proposition}

The back square of this diagram is \(C_2\)-equivariant: by assumption, the twist isomorphisms of \(A \smashover{R} A\) and \(B \smashover{R} B\) are compatible with the action on \(\Sksig[k]\). This allows us to glue this diagram together with its mirror image along the back square, getting a cubical diagram. Moreover, this diagram will have a \emph{\(C_2\)-action}: writing this cube as a functor \(F\) from \(\{0 \to 1\}^3\) to spaces,
there is an isomorphism \(\phi\co F(a,b,c) \to F(b,a,c)\) of cubical diagrams satisfying \(\phi^2 = 1\).

\begin{corollary}\label{cor:cupkcube}
  A central cup-\(k\)-algebra gives rise to a commutative diagram
  \[
  \begin{tikzcd}[row sep=small, column sep=small]
  \Sksig[k] \ar[dr] \ar[rr] \ar[dd]
  && D(k\sigma)' \ar[dd] \ar[dr]\\
  & D(k\sigma)''
  &&\Map_R(A \smashover{R} A, A) \ar[dd]
   \ar[ll,crossing over,leftarrow]\\
  \Map_R(B \smashover{R} B, B) \ar[dr] \ar[rr]
  && \Map_R(B \smashover{R} A, B) \ar[dr]\\
  &\Map_R(A \smashover{R} B, B) \ar[rr]
   \ar[uu,crossing over,leftarrow]
  &&\Map_R(A \smashover{R} A, B)
  \end{tikzcd}
  \]
  Moreover, this diagram has a \(C_2\)-action, via the maps between function spaces induced by twist isomorphisms and the \(C_2\)-action on spheres and disks.
\end{corollary}

\subsection{The transgression}\label{sec:cupktransgression}

Our goal in this section is to prove the following result, analogous to the transgression. 

\begin{proposition}\label{prop:killpowers}
  Suppose that \(A\) is a cup-\(k\) \(R\)-algebra, and that \(B\) is a central cup-\((k-1)\) algebra over \(A\) with unit \(\eta\). If \(\alpha\co X \to A\) is a map of \(R\)-modules such that \(\eta \circ \alpha\) is nullhomotopic, then the induced map \(\Sksig[(k+1)] \smashover{C_2} (X \smashover{R} X) \to B\) is trivial.
\end{proposition}

\begin{proof}
  Choose a nullhomotopy of \(\alpha\), in the form of a factorization \(X \to CX \to B\) of \(\eta \circ \alpha\) where \(CX\) is contractible. This gives rise to a cubical diagram of \(R\)-modules with a \(C_2\)-action:
  \[
    \begin{tikzcd}[row sep=small]
      X \smashover{R} X \ar[rr] \ar[dd] \ar[dr] &&
      CX \smashover{R} X \ar[dd] \ar[dr] \\
      & X \smashover{R} CX&&
      CX \smashover{R} CX \ar[dd]
      \ar[ll,crossing over,leftarrow] \\
      A \smashover{R} A \ar[rr] \ar[dr] &&
      B \smashover{R} A \ar[dr]\\
      & A \smashover{R} B \ar[rr] \ar[uu,crossing over,leftarrow] &&
      B \smashover{R} B
    \end{tikzcd}
  \]
  We can then apply \(\Map_R(-,B)\) and stack with the cube from Corollary~\ref{cor:cupkcube}. This gives a cubical diagram with a \(C_2\)-action:
  \[
    \begin{tikzcd}[column sep=small, row sep=small]
      \Sksig[k] \ar[dr] \ar[rr] \ar[dd]
      && D(k\sigma)' \ar[dd] \ar[dr]\\
      & D(k\sigma)''
      &&\Map_R(A \smashover{R} A, A) \ar[dd]
      \ar[ll,crossing over,leftarrow]\\
      \Map_R(CX \smashover{R} CX, B) \ar[dr] \ar[rr]
      && \Map_R(CX \smashover{R} X, B) \ar[dr]\\
      &\Map_R(X \smashover{R} CX, B) \ar[rr]
      \ar[uu,crossing over,leftarrow]
      &&\Map_R(X \smashover{R} X, B)
    \end{tikzcd}
  \]
  Note that \(CX \smashover{R} X\) and \(CX \smashover{R} CX\) are contractible, making three terms on the bottom of this cube into contractible spaces.
  
  A commutative square has a natural map from the homotopy pushout to the final object. Applying this to the top and bottom faces of this cubical diagram gives us the horizontal maps in the following commutative square of \(C_2\)-equivariant spaces:
  \[
    \begin{tikzcd}
      \hocolim(D(k\sigma)' \from \Sksig[k] \to D(k\sigma)'') \ar[r] \ar[d] &
      \Map_R(A \smashover{R} A, A) \ar[d]\\
      \hocolim(\ast \from \ast \to \ast) \ar[r] &
      \Map(X \smashover{R} X, B).
    \end{tikzcd}
  \]
  This is re-expressed as a homotopy commutative diagram of \(C_2\)-spaces
  \[
    \begin{tikzcd}
      \Sksig[(k+1)] \ar[r] \ar[d] &
      \Map_R(A \smashover{R} A, A) \ar[d]\\
      \ast \ar[r] &
      \Map(X \smashover{R} X, B).
    \end{tikzcd}
  \]
  The top map defines the cup-\(k\) algebra structure on \(A\), while the right-hand one is \(f \mapsto \eta \circ f \circ (\alpha \wedge \alpha)\). Taking adjoints gives a homotopy commutative diagram
  \[
    \begin{tikzcd}
      \Sksig[(k+1)] \smashover{C_2} (X \smashover{R} X) \ar[r] \ar[d] &
      \Sksig[(k+1)] \smashover{C_2} A \smashover{R} A \ar[r]
      & A \ar[d] \\
      \ast \ar[rr] &&
      B.
    \end{tikzcd}
  \]
  This is precisely a trivialization as desired.
\end{proof}

\begin{corollary}\label{cor:killingq}
Let \(A\) and \(B\) be as in Proposition~\ref{prop:killpowers}.
  For \(\alpha\co S^n \to A\) such that \(\eta \circ \alpha\) is nullhomotopic, the map \(\eta Q(\alpha)\co \Sigma^n\RP_{n}^{n+k} \to B\) is trivial.
\end{corollary}

\begin{corollary}\label{cor:killingq1}
  Suppose that \(R\) is a commutative ring spectrum and \(B\) is an associative \(R\)-algebra, with unit \(\eta\co R \to B\). If \(n\) is even or \(2=0\) in \(R\), and \(\alpha \in \pi_n(R)\) satisfies \(\eta(\alpha) = 0\), then \(\eta(Q_1(\alpha)) = 0\).
\end{corollary}

\printbibliography
\end{document}